\documentclass[11pt,letterpaper]{amsart}
\title[Universal Liouville action as a Renormalized volume]{Universal Liouville action as a renormalized volume and its gradient flow}
\author[Bridgeman]{Martin Bridgeman}
\address{Boston College, Chestnut Hill, MA, USA}
\email{\protect\url{bridgem@bc.edu}}
\author[Bromberg]{Kenneth Bromberg}
\address{University of Utah, Salt Lake City, UT, USA}
\email{\protect\url{bromberg@math.utah.edu}}
\author[Vargas Pallete]{Franco Vargas Pallete}
\address{Institut des Hautes \'Etudes Scientifiques, Bures-sur-Yvette, France}
\email{\protect\url{vargaspallete@ihes.fr}}
\author[Wang]{Yilin Wang}
\address{Institut des Hautes \'Etudes Scientifiques, Bures-sur-Yvette, France}
\email{\protect\url{yilin@ihes.fr}}

\usepackage[T1]{fontenc}
\usepackage{comment}

\usepackage[latin9]{inputenc}

\usepackage{mathtools}
\usepackage{lmodern}
\usepackage{amsthm,amsmath,amsfonts,amssymb}
\usepackage{graphicx}
\usepackage{enumerate,enumitem}
\usepackage{authblk}
\usepackage{cite,url}
\usepackage[normalem]{ulem}
\usepackage{calrsfs}

\usepackage{multirow}
\usepackage{array}
\newcolumntype{P}[1]{>{\centering\arraybackslash}p{#1}}
\usepackage{longtable,booktabs}

\usepackage{hyperref}
\hypersetup{
    colorlinks=true, 
    linkcolor=black, 
    urlcolor=black, 
    citecolor=black,
    linktoc=all 
}

\usepackage[activate={true,nocompatibility},final,tracking=true,kerning=true,spacing=true,factor=1100,stretch=20,shrink=20]{microtype}

\microtypecontext{spacing=nonfrench}

\allowdisplaybreaks

\newtheorem{thm}{Theorem}[section]
\newtheorem{cor}[thm]{Corollary}

\newtheorem{lemma}[thm]{Lemma}
\newtheorem{prop}[thm]{Proposition}

\theoremstyle{definition} 
\newtheorem{df}[thm]{Definition}
\newtheorem{ex}[thm]{Example}
\newtheorem{remark}[thm]{Remark}

\numberwithin{equation}{section}

\usepackage[font=normal]{caption}
\usepackage{floatrow}

\usepackage{mathtools}

\usepackage[dvipsnames]{xcolor}

\global\long\def\ii{\mathfrak{i}}

\global\long\def\jj{\mathfrak{j}}

\newcommand{\abs}[1]{\left\lvert #1 \right \rvert}

\newcommand{\brac}[1]{\left \langle #1 \right \rangle}
\newcommand{\norm}[1]{\lVert #1 \rVert}

\newcommand{\mc}[1]{\mathcal{#1}}
\newcommand{\m}[1]{\mathbb{#1}}

\newcommand{\mf}[1]{\mathfrak{#1}}

\def\ie{i.e., }

\renewcommand\Re{\operatorname{Re}}

\def\WP{\operatorname{WP}}

\def\I{\mathrm{I}}
\def\II{\mathrm{I\!I}}
\def\III{\mathrm{I\!I\!I}}

\def\a{\alpha}
\def\b{\beta}
\def\g{\gamma}
\def\G{\Gamma}
\def\d{\delta}

\def\z{\zeta}
\def\t{\theta}

\def\S{\Sigma}

\def\O{\Omega}
\def\vare{\varepsilon}

\def\Chat{\hat{\m{C}}}

\def\dd{\mathrm{d}}

\def\1{\mathbf{1}}

\usepackage{euscript}

\renewcommand{\bold}[1]{\medskip \noindent {\bf \boldmath #1
                        }\nopagebreak[4]}

\newcommand{\zbar}{{\overline{z}}}

\newcommand{\normal}{\overrightarrow{\eta}}

\newcommand{\Diff}{\operatorname{Diff}}
\newcommand{\dist}{\operatorname{dist}}

\newcommand{\id}{\operatorname{id}}

\newcommand{\Isom}{\operatorname{Isom}}

\renewcommand{\Re}{\operatorname{Re}}

\newcommand{\supp}{\operatorname{supp}}

\newcommand{\tr}{\operatorname{tr}}
\newcommand{\vol}{\operatorname{vol}}

\newcommand{\calC}{{\mathcal C}}

\newcommand{\calP}{{\mathcal P}}

\renewcommand{\hbar}{\bar{{\mathbb H}}^3}

\newcommand{\CC}{\mathbb C}
\newcommand{\Sph}{\mathbb S}
\newcommand{\R}{\mathbb R}
\newcommand{\D}{\mathbb D}

\newcommand{\Hs}{{\mathbb H}^3}

\newcommand{\mob}{\mbox{M\"ob}}
\newcommand{\qs}{\mbox{QS}({\mathbb S}^1)}

\newcommand{\pslt}{\mathsf{PSL}_2(\mathbb C)}
\newcommand{\pslr}{\mathsf{PSL}_2(\mathbb R)}
\newcommand{\psu}{\mathsf{PSU}_{1,1}}
\newcommand{\Ep}{\operatorname{Ep}}

 \newcommand{\splus}{{\scriptstyle +}}
 \newcommand{\sminus}{{\scriptstyle -}}

 \def \1{\mathbf{1}}

\def\Id{\operatorname{Id}}
\def\Liouville{\mathbf{S}}

\begin{document}

\maketitle
\begin{abstract}
The universal Liouville action (also known as the Loewner energy for Jordan curves) is a K\"ahler potential on the Weil--Petersson universal Teichm\"uller space, which is identified with the family of Weil--Petersson quasicircles via conformal welding. Our main result shows that, under regularity assumptions, the universal Liouville action equals the renormalized volume of the
hyperbolic $3$-manifold bounded by the two Epstein--Poincar\'e surfaces associated with the quasicircle. 
We also study the gradient descent flow of the universal Liouville action for the Weil--Petersson metric and show that the flow always converges to the origin (the circle). This provides a bound of the Weil--Petersson distance to the origin by the universal Liouville action.
\end{abstract}

\tableofcontents

\section{Introduction}

The Riemann sphere $\Chat$ is the conformal boundary of the hyperbolic $3$-space $\m H^3$. In \cite{epstein-envelopes} 
C.~Epstein gave a natural way to associate with each conformal metric on $\Chat$ a surface in $\m H^3$. In more recent work, these Epstein surfaces have been used to define the renormalized volume of hyperbolic $3$-manifolds, which has deep connections to Teichm\"uller theory of Riemann surfaces and Liouville theory in mathematical physics \cite{KrasnovSchlenker_CMP,krasnov2000holography,TT_Liouville}. We will recall the basics on Epstein surfaces in Section~\ref{sec:epstein}. 

In this work, we define and study the renormalized volume for the universal Teichm\"uller space, which can be identified with the set of quasicircles on the Riemann sphere up to conformal automorphisms.

For this, consider a Jordan curve $\g \subset \Chat$. We let $\O$ and $\O^*$ be the two connected components of $\Chat \smallsetminus \g$.  
Let $\Ep_{\O} : \O \to \m H^3$ be the Epstein map associated with the Poincar\'e  (hyperbolic) metric $\rho_{\O}$ in $\O$,  similarly for $\Ep_{\O^*} : \O^* \to \m H^3$. The maps $\Ep_{\O}$, $\Ep_{\O^*}$ are smooth, extend continuously to the identity map on $\g$, and are immersions almost everywhere.  
We call their images as the Epstein--Poincar\'e surfaces $\S_\O$ and  $\S_{\O^*}$. 
In particular, we note that, unlike in the cases previously considered (see \cite{KrasnovSchlenker_CMP,bridgeman2021weilpetersson}), these Epstein--Poincar\'e surfaces are non-compact and not necessarily embedded and have an infinite area. 
We show the following results.
\begin{prop}[See Proposition~\ref{prop:disjoint}]\label{prop:intro_disjoint}
   If $\g$ is not a circle, then the two Epstein--Poincar\'e surfaces $\S_{\O}$ and $\S_{\O^*}$ are disjoint.
\end{prop}
It follows directly from the definition of Epstein--Poincar\'e map that if $\g$ is a circle, then both $\S_\O$ and $\S_{\O^*}$ are the totally geodesic plane bounded by $\g$ with opposite orientation. See Lemma~\ref{lem:example_disk}.

\begin{prop}[See Corollary~\ref{cor:embedding_wp}] \label{prop:intro_AC}
    When $\g$ is asymptotically conformal (see Theorem~\ref{thm:AC} for several equivalent definitions), there is a neighborhood of $\g$ in $\Chat$ on which the Epstein--Poincar\'e maps $\Ep_\O$ and $\Ep_{\O^*}$ are immersions and embeddings which extend to the identity map on $\g$.
\end{prop}
Quasicircles are in natural correspondence with points in the universal Teichm\"uller space $T(1)$, where we identify a quasicircle with its conformal welding homeomorphism. 
We are interested in a special class of quasicircles, \ie Weil--Petersson quasicircles, which corresponds to the Weil--Petersson universal Teichm\"uller space $T_0(1)$. This space has been studied extensively for it is the connected component of the \emph{unique} homogeneous K\"ahler metric on $T(1)$ (\ie the Weil--Petersson metric) \cite{TT06}, and has a large number of equivalent descriptions from very different perspectives, see, e.g., \cite{bishop-wp,shen13,W2,W3,cui00,johansson2021strong,michelat2021loewner}.

Weil--Petersson quasicircles are asymptotically conformal, so Propositions~\ref{prop:intro_disjoint} and \ref{prop:intro_AC} allow us to define the signed volume between $\S_\O$ and $\S_{\O^*}$. A priori, this volume takes value in $(-\infty, \infty]$ (see Section~\ref{sec:vol_between} for more details). However, we show the following result.
\begin{thm}
    If $\g$ is a Weil--Petersson quasicircle, then the signed volume between the two Epstein--Poincar\'e surfaces,  denoted as $V(\g)$, is finite. 
\end{thm}
See Proposition~\ref{prop:secondorderplane} for the proof for smooth Jordan curves. The result for general Weil--Petersson quasicircles is obtained via an approximation argument, see Corollary~\ref{cor:finite_v_general}.

Since $T_0(1)$ has a \emph{unique} homogeneous K\"ahler structure, its K\"ahler potential is of critical importance.
Takhtajan and Teo defined the \emph{universal Liouville action} $\Liouville$ on $T_0(1)$ and showed it to be such a K\"ahler potential \cite{TT06} (the same functional also appears in probability theory that we will discuss later). 

In this work, we will consider the universal Liouville action as defined for Jordan curves (see Section~\ref{sec:action}) and denote 
 it as $\tilde \Liouville$ for clarity. 
 The functional $\tilde \Liouville (\g)$ can actually be defined for arbitrary Jordan curves, but it is finite if and only if $\g$ is a Weil--Petersson quasicircle. Moreover, $\tilde \Liouville $ is invariant under the action of M\"obius transformations on $\Chat$ (\ie under the $\pslt$ action). As the $\pslt$ action extends to orientation preserving isometries of $\m H^3$, it is very natural to search for a characterization of the class of Weil--Petersson quasicircles and an expression of $\tilde \Liouville$ in terms of geometric quantities in $\m H^3$.

 A pioneering work of C.~Bishop \cite{bishop-wp} shows that the family of Weil--Petersson quasicircles can be characterized as Jordan curves bounding minimal surfaces in $\m H^3$ with finite total curvature. 
We obtain the following similar characterization in terms of Epstein--Poincar\'e surfaces. See also Section~\ref{sec:comments} where we compare Epstein--Poincar\'e surfaces to minimal surfaces and the convex hull, answering a question of Bishop~\cite{BishopQ}.

In fact, the Epstein maps come with a well-defined unit normal $\vec n$ pointing away from $\O$ and from $\O^*$, respectively.  
The mean curvature $H := \tr (B)/2$ is defined using the shape operator $B (v) := -\nabla_v \vec n$. 
\begin{thm}[See Corollary~\ref{cor:WP_HdA}]
    For all Jordan curves,
    $$\int_{\S_\O} H \,\dd a  = \int_{\S_\O} |H \,\dd a |= \int_{\S_\O} |\det B \,\dd a| = \int_{\m D} |\mc S(f) (z)|^2 \frac{(1- |z|^2)^2}{4} \dd^2 z $$
    where $f : \m D \to \O$ is any conformal map, $\mc S(f) = f'''/f' - (3/2) (f''/f')^2$ is the Schwarzian derivative of $f$, $\dd a$ is the area form induced from $\m H^3$ and the Epstein maps, and $\dd^2 z$ is the Euclidean area form.
    
    In particular, $\S_\O$ has finite total mean curvature \textnormal(and finite total curvature\textnormal) \emph{if and only if} $\g$ is a  Weil--Petersson quasicircle. 
\end{thm}

Before this work, no exact identity was known between the K\"ahler potential and geometric quantity in $\m H^3$. The main result of this work is to provide such an identity.

\begin{df} Let $\g$ be a Weil--Petersson quasicircle.
We define the renormalized volume (or W-volume) associated with $\g$ as 
$$V_R(\g) : = V (\g) - \frac{1}{2} \int_{\S_\O \cup \S_{\O^*}} H \dd a \in (-\infty, \infty).$$
\end{df}
The definition is reminiscent of the renormalized volume\footnote{Renormalized volume of a convex co-compact hyperbolic $3$-manifold is sometimes referred to the difference between the volume and half of the boundary area defined through a foliation near the ends. Our formula is closer to that of the \emph{W-volume}. However, in the convex co-compact case, they only differ by a multiple of Euler characteristics of the conformal boundary \cite[Lem.\,4.5]{KrasnovSchlenker_CMP}.} for quasi-Fuchsian manifolds \cite{KrasnovSchlenker_CMP,TT_Liouville}. But we emphasize again that $\S_\O$ and $\S_{\O^*}$ are non-compact, so the analysis involves additional technicalities.

\begin{thm}[See Corollary~\ref{cor:smooth_identity} and Theorem~\ref{thm:general_ineq}] \label{thm:intro_main}
    If $\g$ is a  $C^{5,\a}$ Jordan curve  with $\a > 0$, then
    \begin{equation}\label{eq:intro_main_eq}
    \tilde \Liouville (\g) = 4 V_R(\g).
    \end{equation}
    If $\g$ is a Weil--Petersson quasicircle, then  $\tilde \Liouville (\g) \ge 4 V_R(\g)$.
\end{thm}
Let us comment briefly on the proof of this theorem. 
It is easy to check that when $\g$ is a circle, both sides of \eqref{eq:intro_main_eq} are zero. Under regularity assumptions, we show that both sides' first variations are equal. The variation of $\tilde \Liouville$ was proved in \cite{TT06}, which we recall in Theorem~\ref{thm:S_1_first_variation}. 
The first variation of $V_R$ is more laborious since the Epstein--Poincar\'e surfaces are not compact and are immersed only almost everywhere. 
After administering appropriate truncation (where we make use of the regularity assumption), we re-derive the Schl\"afli formula which expresses the variation of $V_R$ in terms of the mean curvature $H$, the metric $\I$ and the second fundamental form $\II$ on Epstein surfaces (Theorem~\ref{thm:Schl\"afliformula} and Theorem~\ref{thm:schlafli_imm}, with some of the technical details in Section~\ref{gen-Schl\"afli}), then translate the variation formula into quantities defined directly on $\O, \O^* \subset \Chat$ (Theorem~\ref{thm:first_var_VR} and Corollary~\ref{cor:variation_VR_mu}).

For a general Weil--Petersson quasicircle $\g$, we use an approximation by equipotentials (they are analytic curves, and the universal Liouville action increases to that of $\g$). We believe the identity \eqref{eq:intro_main_eq} also holds for a general Weil--Petersson quasicircle. However, our approximation argument only implies the inequality due to the lack of continuity for the volume between the Epstein--Poincar\'e surfaces, see Section~\ref{sec:approximation}.

\bigskip

The second topic of this work concerns the gradient descent flow of $\Liouville$ for the Weil--Petersson metric. We proceed similarly as in Bridgeman--Brock--Bromberg \cite{bridgeman2021weilpetersson}.  
For $[\mu] \in T(1)$ there is a natural isomorphism between the tangent space $T_{[\mu]}T(1)$ and the space $\Omega^{-1,1}(\D^*)$ of harmonic Beltrami differentials on $\D^*$.

\begin{thm} [See Theorem~\ref{thm:gradient}] 
The negative gradient of $\Liouville$ on $T_0(1)$ with respect to the Weil--Petersson metric is the vector field
$$V([\mu]) := -4\frac{\overline{\mc S(g_{\mu}})}{\rho_{\D^*}} \in \Omega^{-1,1}(\D^*).$$
Moreover, the gradient descent flow of $\Liouville$ starting from any point in $T_0(1)$ converges to the origin $[0]$ corresponding to the round circle.
\end{thm}
In fact, we also show that the flow starting from any point in $T(1)$ using the vector field $V$ exists for all time. But here, $V$ cannot be interpreted as the gradient of $\Liouville$ if $[\mu] \notin T_0(1)$, and we do not know the limit and think it is an interesting question.
Using the gradient flow, we also obtain bounds of the Weil--Petersson distance on $T_0(1)$ in terms of the universal Liouville action.

\begin{thm}[See Theorem~\ref{thm:bound_distance}]
There exist universal positive constants $c$ and $K$ such that for all $[\mu] \in T_0(1)$, 
$ c(\dist_{\WP}([\mu],[0]) - K c) \leq \Liouville ([\mu])$.
\end{thm}

\bigskip

Finally, let us make a few remarks on the motivation behind this work and additional comments on the relation with previous works.

Rohde and the last author introduced the \emph{Loewner energy} for Jordan curves \cite{W1,RW}, which is originally motivated by the large deviation theory of random fractal curves Schramm-Loewner evolutions (SLE) \cite{W1,Wang_survey}.  It is shown in \cite{carfagnini2023onsager} that the Loewner energy is the \emph{Onsager--Machlup (or the action) functional} of the SLE loop measure.
It turns out quite surprisingly that the Loewner energy equals exactly $\Liouville/\pi$ as proved in \cite{W2}. Since we will not make use of the Loewner theory but only the 
 fact of $\Liouville$ is a K\"ahler potential on $T_0(1)$, we adopt the terminology of \emph{universal Liouville action} here.
  SLEs play a central role in the emerging field of two-dimensional random conformal geometry. In particular, they provide a mathematical description of the geometric patterns in the scaling limits of 2D critical lattice models \cite{Schramm2000,LSW04LERWUST,Smi:ICM} and 2D conformal field theory (CFT) \cite{BB:CFTSLE,kangMakarov,Dub_SLEVir1,Peltola}.
On the other hand, $\m H^3$ is the Riemannian analog of AdS$_3$ space.
  Our main result Theorem~\ref{thm:intro_main} can be interpreted as the \emph{holography} of the Loewner energy that is reminiscent of the conjectural AdS$_3$/CFT$_2$ correspondence pioneered by Maldacena \cite{Maldacena} (see also, e.g., \cite{Witten_ads,NRT_holographic}). 
  The authors are not aware of a (even conjectural) holographic principle for SLE nor for random conformal geometry in general, this work may be a first step towards this direction. 
  We also mention \cite{Kenyon_02_det} gives a holographic expression for determinants of discrete Dirac operator on periodic bipartite isoradial graphs.

Renormalized volume as a Liouville action has been previously studied for convex co-compact group actions in $\mathbb{H}^3$ (see work by Takhtajan--Teo \cite{TT_Liouville} and Krasnov--Schlenker \cite{KrasnovSchlenker_CMP}), or equivalently, for conformally compact hyperbolic metrics. Applications of this study include bounds for the hyperbolic volume of the mapping tori of pseudo-Anosov maps in terms of their Weil--Petersson translation length or their entropy (by Kojima--McShane \cite{KojimaMcShane18}, see also Brock-Bromberg \cite{BBinflexvol}). This uses a bound (by Schlenker \cite{Schlenker13}) for renormalized volume in terms of Weil--Petersson distance by studying the gradient of the Liouville action, similar to our bound in Theorem \ref{thm:bound_distance}. Moreover, we show in Theorem~\ref{thm:gradient} that every flowline of the gradient converges to the global minimum, in analogy to the result done by the first three authors \cite{BridgemanBrombergVargas} for the relatively acylindrical case. This builds on work by the first two authors and Brock \cite{bridgeman2021weilpetersson}, where they used the gradient flow to find the minimum of renormalized volume for a boundary incompressible hyperbolic $3$-manifold. 

\bigskip

The paper is organized as follows: In Section~\ref{sec:T_1}, we collect the basics about universal Teichm\"uller space, its K\"ahler geometry, characterizations of the Weil--Petersson universal Teichm\"uller space, and the universal Liouville action. 
In Section~\ref{sec:epstein}, we recall the definition of Epstein surfaces and the correspondence between geometric quantities on the surface and those on the conformal boundary. We also prove the immersion and embeddedness of the Epstein--Poincar\'e surfaces associated with an asymptotically conformal Jordan curve. In Section~\ref{sec:V_R}, we study the relation between the two  Epstein--Poincar\'e surfaces associated with the same curve. We show that they are disjoint (except for a circle), and that if the curve is regular enough, the signed volume between the Epstein--Poincar\'e surfaces is finite. 
In Section~\ref{section:LoewnerVR}, we prove the variational formula for the renormalized volume and prove Theorem~\ref{thm:intro_main}. Section~\ref{sec:gradient} is independent of Sections~\ref{sec:epstein}, \ref{sec:V_R}, and \ref{section:LoewnerVR} and deals with the gradient flow of the universal Liouville action. Similarly, Section \ref{sec:comments} describes the relative position of Epstein--Poincar\'e surfaces to minimal surfaces and convex hull. Section~\ref{gen-Schl\"afli} collects the technical details and proves the Schl\"afli formula for the volume bounded by non-immersed Epstein--Poincar\'e surfaces.

\bigskip
\subsection*{Acknowledgments}
We thank Jean-Marc Schlenker and Andrea Seppi for the useful discussion and the anonymous referees for constructive suggestions. We also thank Christopher Bishop and Alexis Michelat for their comments and questions on an earlier version of the manuscript.
M.B. is supported by NSF grant DMS-2005498. K.B. is supported by NSF grant DMS-1906095. F.V.P. is partially supported by NSF grant DMS-2001997 and by the European Union (ERC, RaConTeich, 101116694)\footnote{\label{ERCfootnote}Views and opinions expressed are however those of the authors only and do not necessarily reflect those of the European Union or the European Research Council Executive Agency. Neither the European Union nor the granting authority can be held responsible for them.}. Y.W. is partially supported by NSF grant DMS-1953945 and by the European Union (ERC, RaConTeich, 101116694)\textsuperscript{\ref{ERCfootnote}}.

\section{Universal Weil--Petersson Teichm\"uller space} \label{sec:T_1}

\subsection{Universal Teichm\"uller space}
 \label{sec:universal_T}
We first briefly recall a few equivalent descriptions of the universal Teichm\"uller space $T(1)$.  Let $\Chat = \m C \cup \{\infty\}$, $\D = \{z\ ,\ |z| < 1\}$, $\D^* = \hat\CC-\overline{\D}$ and  $\Sph^1 = \partial \D$. 
The group of orientation preserving conformal automorphism of $\Chat$ is 
$$\mob (\Chat) = \pslt  = \left \{ A
= \begin{pmatrix} a & b \\ c & d
\end{pmatrix} \colon a,b,c,d \in \m C, \, ad - bc = 1\right\}_{/A\sim - A}$$
which acts on $\Chat$ by M\"obius transformations $z \mapsto \dfrac{az+b}{cz+d}$.
The subgroup preserving $\Sph^1$ is 
$$\mob (\Sph^1) = \psu = \left \{ A
= \begin{pmatrix} \a & \b \\ \bar \b & \bar \a
\end{pmatrix} \colon \a, \b \in \m C, \, |\a|^2 - |\b|^2 = 1\right\}_{/A\sim - A}$$
which is isomorphic to $\pslr$.
We will use several equivalent descriptions of $T(1)$ below.

\bigskip 
\noindent {\bf Quasisymmetric maps:} 
We write $\qs$ for the group of sense-preserving quasisymmetric homeomorphisms of $\Sph^1$. The {\em universal Teichm\"uller space} is 
$$T(1) := \mob(\Sph^1) \backslash \qs \simeq \{\varphi \in \qs, \,\varphi\ \mbox{fixes $-1,-\ii$ and $1$}\}.$$
 $T(1)$ is endowed with a group operation given by the composition, and the origin is the identity map $\Id_{\Sph^1}$.
 
\bigskip 

\noindent {\bf Beltrami Differentials:} 
Given a Beltrami differential 
$$\mu \in L_1^\infty(\D^*) = \{ \mu \in L^\infty(\D^*), \,||\mu||_\infty < 1\},$$ we extend it to $\hat\CC$ by reflection, \ie define for $z \in \D$,
$$\mu(z) = \overline{\mu\left(\frac{1}{\overline{z}}\right)}\frac{z^2}{\overline{z}^2}.$$ 
Let $w_\mu : \Chat \to \Chat$ be the solution to the Beltrami equation $\partial_{\overline z}w_{\mu} = \mu \partial_z w_\mu$ fixing $-1,-\ii$ and $1$. Then $w_\mu$ preserves $\Sph^1$ and $w_\mu|_{\Sph^1} \in \qs.$ 
Since every quasisymmetric  circle homeomorphism can be extended to a quasiconformal self-map of $\overline {\m D}$, 
$$T(1) = L_1^\infty(\D^*)/_\sim$$
where $\mu \sim \nu $ if and only if $ \left.w_\mu\right|_{\Sph^1} =  \left.w_\nu\right|_{\Sph^1}$. 
We denote by $\Phi:  L_1^\infty(\D^*) \rightarrow T(1)$ the projection $\mu \mapsto [\mu]$.
Here the origin corresponds to $[0]$.

\bigskip

\noindent {\bf Univalent maps:}
If instead we extend $\mu$  by $0$ on $\D$ and let $w^\mu$ be the unique solution to $w^\mu_{\overline{z}} = \mu w^\mu_{z}$  fixing  $-1,- \ii$ and $1$, then $w^\mu$ is conformal on $\D$. The map $[\mu] \mapsto w^\mu|_{\m D}$  identifies $T(1)$ with
\begin{equation} \label{eq:T1_univalent}
 \{ f:\D \rightarrow \hat\CC \,|\,\mbox{univalent fixing  $-1,-\ii$ and $1$, extendable to q.c. map of $\hat\CC$}\},
\end{equation}
since $\mu\sim \nu $ if and only if $w^\mu = w^\nu$ on $\D$.  

The \emph{Bers' embedding}  is the map $\beta ([\mu]) := \mc S(f) \in  A_\infty (\m D)$ where
\begin{align*}
    A_\infty (\m D):= \{\phi : \m D \to \m C \text{ holomorphic} \,|\, \sup_{z \in \m D} |\phi (z)| (1-|z|^2)^2  < \infty\}. 
\end{align*}
The origin corresponds to $f = \Id_{\m D}$ and $\beta = 0$.

\bigskip 

\noindent {\bf Quasicircles:} By Riemann mapping theorem, the previous identification also gives 
\begin{equation}\label{eq:T1_quasicircle}
T(1) \simeq \{ \g \text{ quasicircles passing through } -1, -\ii, \text{ and } 1 \}
\end{equation}
by the map $[\mu] \mapsto \g_{\mu} := w^\mu (\Sph^1)$. The origin corresponds to $\g_\mu = \Sph^1$. 
We can recover the quasisymmetric circle homeomorphism from $\g_\mu$ through conformal welding. 
Let $\O$ (resp. $\O^*$) denote the connected components of $\hat \CC \smallsetminus \g_\mu$ where $-1, -\ii, 1$ are in the counterclockwise direction of $\partial \O$ (resp. clockwise direction of $\partial \O^*$). 
Let $f_{\mu} = w^\mu|_{\m D} \colon \m D \to \O$ and $g_{\mu} \colon \m D^* \to \O^*$ be the conformal maps fixing $-1, -\ii, 1$. Then,
$$w_\mu|_{\Sph^1} = g_{\mu}^{-1} \circ f_{\mu} |_{\Sph^1}$$
since $g_{\mu} = w^\mu \circ w_\mu^{-1} |_{\m D^*}$. We call $g_{\mu}^{-1} \circ f_{\mu} |_{\Sph^1}$ the \emph{welding homeomorphism} of the quasicircle $\g_\mu$ passing through $-1, -\ii, 1$.

\subsection{K\"ahler Structure and Weil--Petersson Teichm\"uller space} \label{sec:WP}
We first define the following spaces,
\begin{align*}
    A_\infty(\D^*) &= \{ \phi:\D^*\rightarrow \CC \mbox{ holomorphic} \,|\,  \sup_{\D^*}|\phi|\rho^{-1}_{\D^*} < \infty\},\\
    A_2(\D^*) &= \{ \phi:\D^*\rightarrow \CC \mbox{ holomorphic} \,|\,  \int_{\D^*}|\phi|^2 \rho^{-1}_{\D^*} \,\dd^2 z < \infty\} \subset A_\infty (\D^*),
\end{align*}
where $\rho_{\m D^*} (z) = 4/(1- |z|^2)^2$ is the hyperbolic density function and $\dd^2 z = \dd x \wedge \dd y$ if $z = x + \ii y$. 
The inclusion is shown in \cite[Lem.\,I.2.1]{TT06}.
We define the similar spaces $A_\infty (\m D)$ and $A_{2}(\m D)$ (and also $A_\infty (\O)$ and $A_{2}(\O)$).
We will also use the spaces of harmonic Beltrami differentials defined as
\begin{align*}
    \Omega^{-1,1} (\D^*) &= \{ \dot \nu \in L^\infty(\D^*) \,|\, \dot \nu = \rho^{-1}_{\D^*}\, \overline\phi, \, \phi \in A_\infty(\D^*)\};\\
    H^{-1,1}(\D^*)& =  \{ \dot \nu \in L^\infty(\D^*) \,|\, \dot \nu =  \rho^{-1}_{\D^*}\, \overline\phi, \,\phi \in A_2(\D^*)\} \subset \Omega^{-1,1} (\D^*).
\end{align*}

The universal Teichm\"uller space $T(1)$ has a canonical complex structure such that $\Phi:L^\infty_1(\D^*) \rightarrow T(1)$ is a holomorphic submersion. The holomorphic tangent space at the origin is
$$T_{[0]}T(1) = L^\infty(\D^*)/\ker(D_0\Phi) \simeq \Omega^{-1,1}(\D^*)$$
where 
\begin{align*}
 \ker(D_0\Phi)
& = \{\dot \nu \in L^\infty (\m D^*) \,|\, \int_{\m D^*} \dot \nu \phi = 0, \, \phi \text{ holomorphic and} \int_{\m D^*} |\phi| \dd^2 z  <\infty\} \\
 & = :\mf N (\m D^*)
\end{align*}
 is the space of infinitesimally trivial Beltrami differentials. 

The space $L_1^\infty(\D^*)$ has a natural group structure given by the associated quasiconformal maps. We define
$\lambda = \nu \star \mu^{-1}$  
if $w_\lambda = w_\nu\circ w_\mu^{-1}.$
Thus
$$\lambda  = \left(\frac{\nu-\mu}{1-\overline\mu\nu}
\frac{\partial_z w_\mu}{\overline{\partial_{ z}{w_\mu}}}\right)\circ w_\mu^{-1}.$$

We define $R_\mu$ to be right multiplication by $\mu$ on $L_1^\infty(\D^*)$. This
 descends to give a map $R_{[\mu]}: T(1) \rightarrow T(1)$. 
 Furthermore, the complex structure on $T(1)$ is right-invariant. 
 Therefore, $D_0 R_{[\mu]}: T_{[0]}T(1)\rightarrow T_{[\mu]}T(1)$ is a complex linear isomorphism between holomorphic tangent spaces, and we obtain the identification of $T_{[\mu]}T(1) \simeq \Omega^{-1,1}(\D^*).$

 To define a K\"ahler metric on $T(1)$, one needs to endow $T(1)$ with a Hilbert manifold structure.  It is known since \cite{bowick1987holomorphic} that on the subspace $\mc M = \mob (\Sph^1) \backslash\Diff(\Sph^1)$ there is a unique K\"ahler metric up to a scalar multiple. However, $\mc M$ is not complete under the K\"ahler metric. 
 Takhtajan and Teo extend the Hilbert manifold structure on $T(1)$ by defining the Hermitian metric on the distribution  $\mc D([\mu]) = D_0 R_{[\mu]}(H^{-1,1}(\D^*)) \subset T_{[\mu]} T(1)$ induced from $H^{-1,1}(\D^*)$:
 $$\brac{\dot \mu, \dot \nu} := \int_{\m D^*} \dot \mu  \overline {\dot\nu} \rho_{\m D^*} \dd^2 z, \quad \forall \dot \mu, \dot \nu \in H^{-1,1}(\D^*).$$
 They prove that this distribution is integrable and define $T_0(1)$ to be the connected component containing $[0]$, which is called the \emph{Weil--Petersson Teichm\"uller space}. 
 The Hermitian metric defined above is called the \emph{Weil--Petersson metric}. (One may draw the similarity with the Weil--Petersson metric on Teichm\"uller spaces of a Fuchsian group $\G$ where the integral is over $\m D^*/\Gamma$.)
In terms of the four equivalent definitions of $T(1)$, the subspace $T_0(1)$ is characterized as follows:

\bigskip 

\noindent {\bf Quasisymmetric maps}: Y.~Shen \cite{shen13} showed $\varphi \in T_0(1)$ if and only if  $\varphi$ is absolutely continuous with respect to the arclength measure, and $\log \varphi' \in H^{1/2}(\Sph^1)$, namely the fractional Sobolev space of functions $u$ such that
    \begin{equation}\label{eq:H1/2_WP}
    \norm{u}_{H^{1/2}}^2 := \iint_{\Sph^1 \times \Sph^1} \abs{\frac{u(\zeta) - u(\xi)}{\zeta - \xi}}^2 \,\dd \zeta \dd \xi < \infty.
\end{equation}

\bigskip 

\noindent {\bf Beltrami Differentials:} It is shown in \cite{TT06} that $[\mu] \in T_0(1)$ if and only if it has a representative $\mu \in L^\infty_1 (\m D^*)$ such that
$$\int_{\m D^*} |\mu (z)|^2 \rho_{\m D^*} (z) \dd^2 z <\infty.$$

\bigskip

\noindent {\bf Univalent maps:} It is shown in \cite[Thm.\,II.1.12]{TT06}  (see also \cite{cui00}) that a univalent function $f : \m D \to \hat \CC$ fixing $-1, -\ii, 1$ and extendable to a quasiconformal map of $\hat \CC$, corresponds to an element of $T_0(1)$ via the identification \eqref{eq:T1_univalent} if and only if the Schwarzian derivative
$$ \mc S(f) := \left(\frac{f''}{f'}\right)'-\frac 12 \left(\frac{f''}{f'}\right)^2$$
satisfies 
\begin{equation}\label{eq:wp_schwarzian}
    \int_{\m D} |\mc S(f)|^2 \rho_{\m D}^{-1} \,\dd^2 z <\infty.
\end{equation}
In other words, the Bers' embedding $\beta ([\mu]) := \mc S(f) \in A_2 (\m D)$.

Furthermore, let $\tilde f = A \circ f$ where $A$ is an M\"obius map sending $\O = f (\m D)$ to a bounded domain (as a priori, $\overline \O$ may contain $\infty$). Then $f \in T_0(1)$ if and only if
\begin{equation}\label{eq:wp_preS}
\int_{\m D} |\mc N (\tilde f)|^2 \dd^2 z < \infty
\end{equation}
where
$\mc N(\tilde f)= \tilde f''/\tilde f'$ is the \emph{nonlinearity} of $\tilde f$. We note that the expression in \eqref{eq:wp_schwarzian} is invariant under the transformation
$f \to A \circ f \circ B$, for all $A \in \pslt$ and $B \in \psu$ but the expression in \eqref{eq:wp_preS} is not invariant under such transformations.

\bigskip

\noindent {\bf Quasicircles:} A quasicircle passing through $-1, -\ii, 1$
which corresponds via \eqref{eq:T1_quasicircle} to an element of $T_0(1)$ is called a \emph{Weil--Petersson quasicircle}.
It is easy to see that if $\g$ and $\tilde \g$ are two quasicircles passing through $-1, -\ii, 1$ and $\tilde \g = A (\g)$ for some $A \in \pslt$, then $\tilde \g$ is Weil--Petersson if and only if $\g$ is Weil--Petersson.
Therefore, we may extend the definition to say that a Jordan curve $\g$ is Weil--Petersson if and only if it is $\pslt$-equivalent to a Weil--Petersson quasicircle passing through $-1, -\ii, 1$.

\subsection{Universal Liouville action}\label{sec:action}
 Takhtajan and Teo introduced the \emph{universal Liouville action} $\Liouville$ on $T_0(1)$ and showed it to be a K\"ahler potential on $T_0(1)$. See \cite[Thm.\,II.4.1]{TT06}. 
 We will consider it as a functional on the space of Weil--Petersson quasicircles.
 
 Indeed, let $\g$ be a Jordan curve that does not pass through $\infty$.  Let $D$ and $D^*$ be respectively the bounded and unbounded connected component of $\hat \CC \smallsetminus \g$,  $f : \m D \to D$ and $g : \m D^* \to D^*$ be \emph{any} conformal maps such that $g (\infty) = \infty$ (note that $D$ might not be $\O$, it can also be $\O^*$, and $f$ and $g$ are different from the canonical maps $f_{\mu}$ and $g_{\mu}$). 
 Define
 \begin{equation}\label{eq:IL_S1}
    \tilde \Liouville (\g) : = \int_\D |\mc N(f)|^2 \, \dd^2 z+ \int_{\D^*} |\mc N(g)|^2 \, \dd^2 z+ 4 \pi \log|f'(0)/g'(\infty)|
    \end{equation}
    and is $\pslt$-invariant (it can be seen via the identity with $\pi$ times the Loewner energy of $\g$ \cite{W2}) and finite if and only if $\g$ is a Weil--Petersson quasicircle. The universal Liouville action $\Liouville ([\mu])$ for $[\mu] \in T_0(1)$ is defined as $\tilde \Liouville (A (\g_\mu))$ where $\g_\mu$ is the Weil--Petersson quasicircle passing through $-1,-\ii, 1$  corresponding to $[\mu]$ via the identification \eqref{eq:T1_quasicircle} and $A \in \pslt$ is any M\"obius transformation such that $A(\g_\mu)$ is bounded.

The universal Liouville action $\Liouville$ satisfies the following properties:
\begin{itemize}[itemsep=-1 pt]
    \item $\Liouville([\mu]) \ge 0$ for all $[\mu] \in T_0(1)$ (see, e.g., \cite[Thm.\,1.4]{W2});
    \item $\tilde \Liouville(\g) = 0$ if and only if $\g$ is a circle, or equivalently, $[\mu] = [0]$.
\end{itemize}

The first variation formula of $\Liouville$ from \cite{TT06} will be a key ingredient in our proofs. 
We now restate it for $\tilde \Liouville$.
Let $\g$ be the Weil--Petersson quasicircle 
passing through $-1,-\ii,1$ corresponding to an element $[\mu]$ of $T_0(1)$. 
Let $\O$ and $\O^*$  be the connected components  of $\hat \CC \smallsetminus \g$ as in Section~\ref{sec:universal_T}. Let  $f_{\mu} : \m D \to \O$ and $g_{\mu}: \m D^* \to \O^*$ be the conformal maps fixing $-1, -\ii, 1$. 
Let $\dot \nu \in  H^{-1,1}(\D^*) \simeq T_{[\mu]}T_0(1)$, $t \in (-\norm{\dot \nu}_\infty^{-1}, \norm{\dot \nu}_\infty^{-1})$, $w_t : \hat \CC \to \hat \CC$ be the solution fixing $-1, -\ii, 1$ to the Beltrami equation 
$$\frac{\partial_{\bar z} w_t}{ \partial_z w_t} (z) = \begin{cases} 0  \qquad & z \in \O,\\
t (g_{\mu})_*\, \dot \nu (z) \, \qquad &z \in \O^*
\end{cases}
$$
where $$(g_{\mu})_* \,\dot \nu (z) = \dot \nu \circ g_{\mu}^{-1} \,\frac{\overline{(g_{\mu}^{-1})'}}{(g_{\mu}^{-1})'}.
$$
 
We let $\g_t = w_t (\g)$ which is a small deformation of $\g$. 

\begin{thm}[\!\!{\cite[Cor.\,II.3.9]{TT06}}]  \label{thm:S_1_first_variation}
The universal Liouville action satisfies the following first variation formula. Let $\dot \nu \in  H^{-1,1}(\D^*) \simeq T_{[\mu]}T_0(1)$,
\begin{align*}
    (\dd \Liouville)_{[\mu]}(\dot \nu) & =  \frac{\dd}{ \dd t}\Big|_{t = 0} \tilde \Liouville(\g_t) = 
4  \Re \int_{\m D^*} \dot \nu \mc S(g_{\mu})  \dd^2 z \\
& = - 4 \Re \int_{\O^*}  ((g_{\mu})_*\dot \nu) \, \mc S(g_{\mu}^{-1}) \,\dd^2 z.
\end{align*}
\end{thm}
\begin{remark}
    We note that compared to the formula in \cite{TT06}, we take derivatives of $\Liouville$ in the real tangent space (which is canonically isomorphic to the holomorphic tangent space) while \cite{TT06} takes derivatives in the holomorphic tangent space and both derivatives are related by
$$(\dd \Liouville)_{[\mu]}(\dot \nu) = 2 \Re \partial_{\dot \nu} \, \Liouville([\mu]). $$
The last equality in Theorem~\ref{thm:S_1_first_variation} follows from a change of variable and the chain rule for Schwarzian derivatives, which shows
$$\mc S(g^{-1}) = - \mc S(g) \circ g^{-1} (g^{-1}{}')^2.$$
See another proof of Theorem~\ref{thm:S_1_first_variation}  in \cite{sung2023quasiconformal} using the Loewner theory and when $\dot \nu$ is compactly supported.
\end{remark}

\begin{remark}\label{rem:var_general_beltrami}
    We choose $\dot \nu$ to be \emph{harmonic} Beltrami differential as $H^{-1,1}(\m D^*)$ is isomorphic to $T_{[\mu]} T_0(1)$.
    Clearly, the variational formula also holds for $\dot \nu \in H^{-1,1} (\m D^*) \oplus \mf N(\m D^*)$ if $\int |\mc S (g)| \dd^2 z < \infty$, which is the case, e.g., whenever the curve $\g$ is $C^{3,\a}$ for $\a > 0$.
\end{remark}

\section{Preliminaries on Epstein surfaces}\label{sec:epstein}
\subsection{Definition of general Epstein surfaces}\label{subsec:generalepstein}
In \cite{epstein-envelopes}, Epstein associated to a smooth conformal metric $\rho$ on a domain $\Omega \subseteq \Sph^n$ a surface $\Ep_\rho:\Omega \rightarrow \m H^{n+1}$ given by taking an envelope of horospheres based at points of $\Omega$ with size determined by the conformal metric. Explicitly, for $x \in  \m H^{n+1}$ in the ball model, we let $\nu_x$ be the hyperbolic visual measure on the unit sphere  $\Sph^n = \partial \m H^{n+1}$ from $x$ (namely, the pull-back of the round metric on $\Sph^n$ by any isometry of $\m H^{n+1}$ sending $x$ to the origin), then for $z\in \Omega$ considering both $\rho$ and $\nu_x$ as conformal metrics, we define
$$\mathfrak H(z,\rho) = \{ x\in \m H^{n+1} \,|\, \nu_x(z) = \rho(z)\}.$$
Then the set $\mathfrak H(z,\rho)$ is a horosphere based at $z$. The \emph{Epstein map} $\Ep_\rho$ is the solution to the envelope equation of these horospheres. More precisely, there exists a (unique) smooth map 
\begin{equation}\label{eq:tilde_Ep}
    \widetilde{\Ep}_\rho\colon \Omega \to T^1\m H^{n+1}
\end{equation}
 such that $\widetilde{\Ep}_\rho(z)$ is an outward pointing normal to the horoball bounded by $\mathfrak H(z,\rho)$ and
$$\Ep_\rho\colon \Omega \to \m H^{n+1}$$
is the composition of $\widetilde\Ep_\rho$ with the projection $T^1\m H^{n+1} \to \m H^{n+1}$ and that the image of the tangent maps of $\Ep_\rho$ at $z$ is orthogonal to $\widetilde\Ep_\rho(z)$. We call the image of $\Ep_\rho$ the \emph{Epstein surface associated with $\rho$} and denote it by $\S_\rho$.

These {\em Epstein surfaces} generalize surfaces such as the convex hull boundary and have had numerous applications in hyperbolic geometry, complex analysis, and the study of univalent functions.
We record some basic facts about Epstein maps. The following property follows directly from the definition.

\begin{lemma}[Naturality of Epstein map]\label{lem:naturality}
    If $x\in\mathbb{H}^{n+1}$, $h\in \Isom_+(\mathbb{H}^{n+1})$ then $h^*(\nu_{h(x)}) = \nu_x$. 
Hence it follows that
 \begin{equation}\label{eq:Eps_invariant}
 \Ep_{\rho} = h \circ \Ep_{h^* \rho}
\end{equation}
where $h^* \rho$ is the pull-back metric of $\rho$ under $h$.
\end{lemma}

\begin{thm}[See \cite{KrasnovSchlenker_CMP,BBB}] \label{thm:basic_epstein}
Let $\Omega$ be a domain in $\Sph^n$ and $\rho$ a smooth conformal metric on $\Omega$. Let $\rho_t = e^{2t}\rho$ for some $t \in \m R$.
\begin{enumerate}
    \item The value of $\Ep_\rho(z)$ is determined by $\rho$ and its first derivatives at $z$.

\item We let $\mathfrak g_t: T_1\m H^3 \rightarrow T_1\m H^3$ be time $t$ geodesic flow.  Then $\widetilde\Ep_{\rho_t} = \mathfrak{g}_{-t} \circ \widetilde\Ep_\rho$. 
    \item Let $\mathfrak g_{-\infty}: T_1 \m H^3 \rightarrow \hat{\m C}$ be the {\em hyperbolic Gauss map} sending a tangent vector to the endpoint of the associated geodesic ray as $t \to -\infty$. Then $g_{-\infty} \left(\widetilde \Ep_\rho (z) \right) = z$.

    \item For each $z \in \Omega$ there are at most two values of $t$ where $\Ep_{\rho_t}$ is not an immersion at $z$.
\end{enumerate}
\end{thm}

Whenever $\Ep_\rho$ is an immersion we pullback the fundamental forms $\I, \II, \III$ of $\Sigma_\rho$ with respect to $\Ep_\rho$ to $\O$ to obtain 
$$\II(X,Y) = \I(BX,Y)\qquad \III(X,Y) = \II(BX,Y) = \I(BX,BY)$$
where $B$ is the pullback of the shape operator of $\Sigma_\rho$, namely,  
$$D\Ep_\rho (B X) = -\nabla_{D\Ep_\rho X} {\widetilde \Ep_\rho}$$
since $\widetilde \Ep_\rho$ defines a unit normal vector field on $\S_\rho$.
The eigenvalues $\{k_+, k_-\}$ of $B$ are the {\em principal curvatures} of the surface $\Sigma_\rho$. The {\em mean curvature} $H$ is defined as $\tr(B)/2$.

If $\I_t$ is the pullback of the metric on $\Sigma_{\rho_t}$ under $\Ep_{\rho_t}$, then by \cite{KrasnovSchlenker_CMP}\footnote{We note that our convention for Epstein maps is slightly different from the one in \cite{KrasnovSchlenker_CMP}, that our foliation $\Sigma_{\rho_t}$ coincides with their foliation $S_{t - (\log 2)/2}$. Our choice is such that when $\rho$ is the Poincar\'e metric on the unit disk, the Epstein surface is exactly the totally geodesic plane bounded by the unit circle. See Lemma~\ref{lem:example_disk}.},
$$\I_t(X, Y) = \I(\cosh(t)X+\sinh(t)BX, \cosh(t)Y+\sinh(t)BY).$$
Expanding out, we obtain
$$\I_t = \frac{1}{4}\left(e^{2t} \hat \I + 2\hat \II + e^{-2t}\hat \III \right)$$
where
\begin{eqnarray*}
\hat \I&=&  \I+2\II +\III  =  \I((\id+B)\cdot,(\id+B)\cdot) \\
\hat \II &=& \I-\III  =  \I((\id+B)\cdot,(\id-B)\cdot) \\
 \hat \III &=& \I-2\II+\III  =   \I((\id-B)\cdot,(\id-B)\cdot).
 \end{eqnarray*}

These are called the {\em fundamental forms at infinity} $\hat \I, \hat \II, \hat \III$  and it is natural then to define the {\em shape operator at infinity} by $\hat B = (\id+B)^{-1}(\id-B)$ which satisfies
$$\hat \II(X,Y) = \hat \I(\hat BX,Y)\qquad \hat \III(X,Y) = \hat \II(\hat BX,Y) = \hat\I(\hat BX,\hat BY).$$
Further, we define the {\em mean curvature at infinity} as $\hat H = \tr(\hat B)/2$. These formulas can be inverted with   $B = (\id+\hat B)^{-1}(\id-\hat B)$ and \begin{eqnarray*}
\I &=& \frac{1}{4}\left(\hat\I+2\hat\II +\hat\III\right) = \frac{1}{4}\hat\I\left((\id+\hat B)\cdot,(\id+\hat B)\cdot\right)\\
 \II &=& \frac{1}{4}\left(\hat\I-\hat\III\right) = \frac{1}{4}\hat\I\left((\id+\hat B)\cdot,(\id-\hat B)\cdot\right)  \\
 \III &=&  \frac{1}{4}\left(\hat\I-2\hat\II+\hat\III\right) = \frac{1}{4}\hat\I\left((\id-\hat B)\cdot,(\id-\hat B)\cdot\right).
 \end{eqnarray*}

If $\Phi = g(z)\, \dd z^2$ is a quadratic differential ($g$ not necessarily holomorphic) and $\rho$ a conformal metric then we define the norm of $\Phi$ with respect to $\rho$ by
\begin{equation}\label{eq:norm_def}
\|\Phi(z)\|_\rho = \frac{|\Phi(z)|}{\rho(z)}.
\end{equation}

Epstein (see \cite[Section 5]{epstein-envelopes}) gave the following description of the fundamental forms at infinity. 
\begin{thm}\label{thm:forms_infty}
Let $\rho = e^{\varphi}|\dd z|^2$ be a conformal metric on $\Omega$. Then
\begin{itemize}
\item $\hat \I = \rho.$
\item If $\hat K$ is the Gaussian curvature of  $\hat \I$ and   $\vartheta = (\varphi_{zz} - \frac{1}{2}\varphi_z^2) \, \dd z^2$, then
$$\hat{\II} = \vartheta  + \overline\vartheta- \hat K\rho.$$
\item The eigenvalues of $\hat B$ are 
$$\hat k_\pm = \frac{1-k_\pm}{1+k_\pm} = -\hat K \pm 2\|\vartheta\|_{\rho}.$$
\item The Epstein map $\Ep_\rho$ is an immersion on
$\{ z\in \O \ |\  -1 \not\in \{\hat k_+,\hat k_-\} \}.$
\end{itemize}
\end{thm}

By the above
\begin{equation}\label{eq:relate_rho_I}
\rho =  \I((\id+ B)\cdot,(\id+B)\cdot) = (\id+B)^*\I.
\end{equation}
We let $\dd \hat a$ be the area form for $\hat\I =\rho$, then the area measure on $\Sigma_\rho$ satisfies

$$
\dd A = \frac{1}{4}\left|\det(\id+\hat B)\right| |\dd \hat a|.$$
We define the {\em signed area} of $\Sigma_\rho$, denoted be $\dd a$, as the area form with induced orientation by ${\widetilde \Ep_\rho}$, which satisfies
$$\dd a = \frac{1}{4}\det(\id+\hat B) \,\dd \hat a.$$
Thus $\dd A  = |\dd a|$.
\begin{cor}\label{cor:epstein_curv} Let $\Sigma_\rho$ be the Epstein surface for  $\rho = e^{\varphi}|\dd z|^2$. Then at places where $\Ep_\rho$ is an immersion, 
\begin{itemize} 
\item $\hat H = -\hat K$,
\item 
$H\dd a = \left(\frac{1-\hat K^2}{4} + \|\vartheta\|^2_\rho\right) \dd \hat a$,
\item
$\det(B)\,\dd a =  \left(\frac{(1+\hat K)^2}{4} - \|\vartheta\|^2_\rho\right) \dd\hat a$.
\end{itemize}
\end{cor}
\begin{proof}
As the eigenvalues of $\hat B$ are $\hat k_\pm = -\hat K \pm 2\|\vartheta\|_\rho$, then
$$\tr(\hat B) = -2\hat K\qquad \det(\hat B) = \hat K^2-4\|\vartheta\|_\rho^2.$$
Thus $\hat H = \tr(\hat B)/2 = -\hat K$. Also  
\begin{eqnarray*}
H &=& \frac{1}{2}(k_1 + k_2) = \frac{1}{2}\left(\frac{1-\hat k_1}{1+\hat k_1}+\frac{1-\hat k_2}{1+\hat k_2}\right) = \frac{1-\det(\hat B)}{\det(\id+\hat B)}\\
\det(B) &=& k_1k_2 = \left(\frac{1-\hat k_1}{1+\hat k_1}\right)\left(\frac{1-\hat k_2}{1+\hat k_2}\right) = \frac{\det(\id-\hat B)}{\det(\id+\hat B)}.
\end{eqnarray*}
From this, we obtain
$$H \dd  a =  \frac{1}{4}(1-\det(\hat B)) \dd \hat a = \left(\frac{1-\hat K^2}{4} + \|\vartheta\|^2_\rho\right) \dd\hat a$$
and
$$\det(B)\,\dd a = \frac{1}{4}\det(\id-\hat B)\dd \hat a = \left(\frac{(1+\hat K)^2}{4} - \|\vartheta\|^2_\rho\right) \dd\hat a$$
as claimed.
\end{proof}

\subsection{Epstein--Poincar\'e surface}\label{subsec:EPsimplyconnected}
We now consider the Epstein map associated with the Poincar\'e metric $\rho_\O$ (namely, complete and $\hat K \equiv - 1$) on a simply connected domain $\O \subsetneq \m C$, that we call the \emph{Epstein--Poincar\'e map} $\Ep_\O$. We write the Epstein--Poincar\'e surface as $\S_\O$ similarly. 
There are two connected components of the complement of a Jordan curve $\g$ in $\Chat$, we will study the relation between the two Epstein--Poincar\'e maps later in Section~\ref{sec:V_R} which will be crucial to defining renormalized volume. 
However, let us first record some properties of a single Epstein--Poincar\'e map.

As the Euclidean diameter of the horosphere of $z \in \O$ associated with  $\rho_\O$ goes to $0$ as $z \to \partial \O$, the Epstein map extends to the identity map on $\partial \O$ (and $\Ep_\O$ meets $\Chat$ along $\partial \O$).
 
Epstein showed that in this case, $\vartheta = \mc S(f^{-1}) $ the Schwarzian quadratic differential of $f^{-1}$, where $f:\mathbb{D}\rightarrow\O$ is any conformal map. 
It follows from above that 
\begin{equation}\label{eq:Poincare_k_hat}
    \hat B \text{ has eigenvalues }1 \pm 2\|\mc S(f^{-1})\|_{\O}
\end{equation} 
where $\|\cdot\|_{\O} = \|\cdot\|_{\rho_\O}$ is the norm with respect to the hyperbolic metric $\rho_\O$ as defined in \eqref{eq:norm_def}.
Thus inverting the principal curvatures are
\begin{equation}\label{eq:k_EP_schwarzian}
k_\pm = -\frac{\|\mc S(f^{-1})\|_\O}{\|\mc S(f^{-1})\|_\O \pm 1}.
\end{equation}

Applying  Corollary \ref{cor:epstein_curv} above, we obtain the following result.
\begin{thm}{\label{thm:PE}}
Let $\O \subsetneq \m C$ be a simply connected domain.
Then $\Ep_\O$ is an immersion on  $\{z \in \O \ | \ \|\mc S(f^{-1})(z)\|_\O \neq 1\}$. Furthermore,
$$ H \dd a = -\det B \,\dd a = \|\mc S(f^{-1})(z)\|_\O^2 \,\dd\hat a = |H \dd a|.$$
\end{thm}
Observe that since $\|\mc S(f^{-1})(z)\|_\O^2 \,\dd\hat a$ is a smooth form defined for all points $z\in\O$ and the set $\{z \in \O \ | \ \|\mc S(f^{-1})(z)\|_\O \neq 1\}$ is dense and has full measure, we can uniquely extend $H\dd a$ to $\O$ as $\|\mc S(f^{-1})(z)\|_\O^2 \,\dd\hat a$ and obtain the following corollary from the characterization \eqref{eq:wp_schwarzian}.
\begin{cor}\label{cor:WP_HdA}
A Jordan curve $\gamma$ is a Weil--Petersson quasicircle if and only if
$$ \int_{\Sigma} H \dd a =- \int_{\Sigma} \det(B) \,\dd a  < \infty.$$
\end{cor}

From the formula of the principal curvatures \eqref{eq:k_EP_schwarzian}, we obtain immediately the following explicit example of Epstein--Poincar\'e surface.
\begin{lemma}\label{lem:example_disk}
The Epstein--Poincar\'e surface associated with $\m D$ \textnormal{(}take $f = \id_{\m D}$\textnormal{)} is the totally geodesic plane bounded by $\partial \m D$ and $\Ep_{\m D} (0) = (0,0,1)$ in the upper half-space model. 
From the naturality of Epstein map, if $M$ is an M\"obius transformation \textnormal{(}which extends to an isometry of $\m H^3$ as we explain below\textnormal{)}, then $\Ep_{M(\m D)} \circ M = M \circ \Ep_{\m D}$. 
\end{lemma}

\begin{remark}\label{rem:E_radius}
    In particular, the Euclidean radius of the horosphere associated with $\rho  = 4 |\dd z|^2$ is $1/2$. From this, we obtain that, more generally, the Euclidean radius of the horosphere associated with $\rho$ is $1/\sqrt {\rho}$.
\end{remark}

Although totally geodesic planes are trivial examples of Epstein--Poincar\'e surfaces, the other Epstein--Poincar\'e surfaces have an elegant description in terms of these maps associated with the geodesic planes and osculating M\"obius transformations (Lemma~\ref{lemma:osc-mob}). 

Let us first recall the classical result about extending an M\"obius transformation to an isometry of $\m H^3$. For this, we use the upper half-space model and use  quaternions to parametrize $\m H^3 = \m C \oplus  \jj \m R_+ = \{z +   \jj t  \,|\, z = x +   \ii y \in \m C, t > 0 \}$ (so that $(x,y,t)$ in the upper-half space is identified with $z +  \jj  t $). A M\"obius transformation $z \mapsto \frac{az + b}{cz+d}$, where $\begin{psmallmatrix}
    a & b \\ c & d
\end{psmallmatrix} \in \pslt$ extends to the isometry of $\m H^3$ by $$ Z \mapsto (aZ + b) (cZ+d)^{-1}, \quad \forall Z = z + \jj t $$
using the multiplication on quaternions, see \cite[Sec.\,2.1]{Ahlfors_Mobius} for more details.

 Since the Epstein map depends on the metric and its derivatives at infinity (Theorem~\ref{thm:basic_epstein}), the Epstein--Poincar\'e map $\Ep_\O \circ f (z_0)$ depends only on the two-jet of $f$ at $z_0$ (namely the values of $f(z_0), f'(z_0)$, and $f''(z_0)$). 
There exists a unique M\"obius transformation $M_{f, z_0}$ with the same two-jet as $f$ at $z_0$,  called the {\em osculating M\"obius transformation of $f$ at $z_0$}. 
 Therefore, 
$$\Ep_\O \circ f(z_0) = \Ep_{M_{f,z_0}(\m D)} \circ M_{f,z_0} (z_0).$$ 
From the naturality of the Epstein map (Lemma~\ref{lem:naturality})  
$$\Ep_{M_{f,z_0}(\m D)} \circ M_{f,z_0} (z_0) = M_{f,z_0} \circ \Ep_{\m D} (z_0).$$ 
Summarizing, we have proved (assuming $z_0 = 0$) and using Lemma~\ref{lem:example_disk}:
\begin{lemma}\label{lemma:osc-mob} Let  $\O \subsetneq \m C$ be a simply connected domain and 
$f:\m D\rightarrow \Omega$ be a univalent map. Then
$\Ep_\O \circ f (0) = M_{f,0} \circ \Ep_{\m D} (0) =  M_{f,0}(\jj).$
\end{lemma}

We now state equivalent conditions for a curve to be asymptotically conformal.

\begin{thm}{(See \cite[Thm.\,11.1]{Pommerenke_boundary})} \label{thm:AC}
Let $f$ be a conformal map from $\m D$ onto a domain bounded by a Jordan curve $\g$.
  The following are equivalent: 
    \begin{enumerate}
         \item[(AC1)] $\g$ is asymptotically conformal;
        \item [(AC2)] \label{it:asymp_conf_P}
$\lim_{|\z| \to 1\sminus} \frac{f''(\zeta)}{f'(\zeta)}(1-|\zeta|^2) = 0$;
\item [(AC3)] \label{it:asym_conf_S}
$\lim_{|\z| \to 1\sminus} \norm{\mc S(f) }_{\m D}(\zeta) = 0$.
    \end{enumerate}
\end{thm}

From now on, we will assume that the boundary of $\O$ is asymptotically conformal and use a few classical results from geometric function theory. 

\begin{ex}
    Weil--Petersson quasicircles satisfy AC3 (see \cite[Corollary II.1.4]{TT06}) and are therefore asymptotically conformal.
\end{ex}

The following is a simple consequence of the Koebe 1/4 theorem.
\begin{thm}{(See \cite[Cor.\,1.4]{Pommerenke_boundary})}\label{thm:qh}
Let $\gamma$ be a Jordan curve bounding $\Omega$ and $\rho_\Omega$ the hyperbolic metric on $\Omega$. Then
$$\frac{1}{2d(z,\gamma)} \leq \sqrt{\rho_\O(z)} \leq \frac{2}{d(z,\gamma)}$$
where $d (\cdot, \cdot)$ denotes the Euclidean distance in $\m R^2$.
\end{thm}

Using this, we obtain the following control over the behavior of the Epstein map.

\begin{cor}\label{cor:bdyvals}
Let $\gamma$ be a Jordan curve bounding $\Omega$. Then $\Ep_\Omega$ extends continuously to the identity on $\gamma$ with
$$(\sqrt{5}-2)d(z,\gamma) \leq d(\Ep_\Omega(z),\gamma) \leq 5d(z,\gamma)$$
where $d (\cdot, \cdot)$ denotes the Euclidean distance in $\m R^3$.
Furthermore if $\Ep_\Omega(z) = (Z(z),\xi(z)) \in \m C\times \m R_+$, then
$$\frac{1}{5}d(z,\gamma) \leq \xi(z) \leq 4d(z,\gamma).$$
  \end{cor}
\begin{proof} 
Let $s = d(z,\gamma)$. Since $\Ep_\Omega(z)$ is on the boundary of a horosphere of Euclidean radius $r = 1/\sqrt{\rho(z)}$ based at $z$ by Remark~\ref{rem:E_radius}, Theorem \ref{thm:qh} shows that
$$s/2 \leq r \leq 2s.$$ 
We let $z_i\rightarrow z\in \gamma$. Then
$$d(\Ep_\Omega(z_i),z) \leq 2r+d(z_i,z) \leq 4d(z_i,\gamma)+d(z_i,z) \leq 5d(z_i,z).$$
Therefore $\Ep_\Omega(z_i)\rightarrow z$ giving $\Ep_\Omega$ extends continuously to the identity on $\gamma.$

By the triangle inequality, we also have
$$ \sqrt{s^2+r^2}-r \leq d(\Ep_\Omega(z),\gamma) \leq 2r+s.$$
Thus 
$$(\sqrt{5}-2)s \leq d(\Ep_\Omega(z),\gamma) \leq 5s.$$

To bound $\xi$, let  $z_0 \in \Omega$ and $f:\m D\rightarrow \Omega$ a uniformizing map with $f(0) = z_0$. Further, by post-composition by a translation, it suffices to consider $z_0= 0$. Thus the osculating M\"obius map $M$ of $f$ at $0$ (i.e. the M\"obius map with the same 2-jet at $0$)  is
$$
M(\z) = \frac{\a \z}{\beta \z + 1/\a}$$
where 
$\a^2 = f'(0)$ and 
$2 \a \b = - f''(0)/f'(0)$.  By Lemma~\ref{lemma:osc-mob} 
$$\Ep_\O \circ f (0) = M(\jj)  =  \a \jj (\b \jj + 1/\a)^{-1} = \frac{\alpha\overline{\beta} + \jj}{|\beta|^2+\frac{1}{|\alpha|^2}}. 
$$
Thus
\begin{equation}\label{eq:xi_0_EP}
\xi(z_0) = 
\frac{|\a|^2}{|\a \beta|^2+1}= \frac{|f'(0)|}{1+\left|\frac{f''(0)}{2f'(0)}\right|^2}.
\end{equation}
As $f$ is univalent, by the Bieberbach theorem (see \cite{Bieberbach}) then $|f''(0)|\leq 4|f'(0)|$. We also know that $1/r^2 = \rho(z_0) = 4 / |f'(0)|^2$. Thus  $|f'(0)| = 2r$ and
$$ \frac{s}{5} \le \frac{2r}{5} = \frac{|f'(0)|}{5} \leq \xi(z_0) \leq |f'(0)| = 2r \le 4 s.$$
The result follows.
\end{proof}

\begin{cor}\label{cor:embedding_wp}
If $\O$ is bounded by an asymptotically conformal curve $\g$ then $\Ep_\Omega$ is an 
embedding in a neighborhood of $\partial \O$.
\end{cor}
\begin{proof}
We note that $\norm{\mc S(f^{-1}) (f(\z))}_{\O} =  \norm{\mc S(f) (\z)}_{\m D}$. 
 Since $\g$ is asymptotically conformal, Theorem~\ref{thm:PE} and (AC3) imply that 
 $\Ep_\O$ is an immersion in a neighborhood of $\partial \O$.

By Corollary \ref{cor:bdyvals}, $\Ep_{\O}$ extends to the identity on $\partial\Omega$. Therefore, if $\Ep_{\O}$ is injective in a neighborhood of $\partial \O$, then by compactness, the extension map is a homeomorphism on a neighborhood of $\partial \O$.
Thus it suffices to show $\Ep_{\O}$ is injective in a neighborhood of $\partial \O$.

If $\Ep_{\O}$
is not injective in a neighborhood of $\partial \O$ then  there exists sequences $x_i \rightarrow u, y_i \rightarrow v$ with $x_i \neq y_i$ but $\Ep_\O(x_i) = \Ep_\O(y_i)$ and $u,v \in \partial 
\O$. Then $u = v$ by Corollary~\ref{cor:bdyvals}. 

We now obtain our contradiction to $\Ep_\O(x_i) = \Ep_\O(y_i)$ for all $i$. We let  $A_\d := \{f(\z)  \in \O \,|\, 1-|\z| < \d \}$. Given any $\varepsilon > 0$ we can choose $\d$ such that $\|\mc S (f^{-1})\| < \varepsilon$ on $A_\d$. 
For sufficiently large $i$  the geodesic arc $\gamma_i$ joining $x_i$ to $y_i$ is in $A_\d$. Furthermore, for any fixed  $r_0> 0$, for sufficiently large $i$ the $r_0-$neighborhood of the geodesic arc $\gamma_i$ is in $A_\d$. Therefore as $\|\mc S (f^{-1})\| < \varepsilon$ on $A_\d$ by \cite[Lemma 3.5]{BBvariation}, then for $r_0 \leq 1/2$ the curve $\Ep_{\O}\circ \gamma_i$  has geodesic curvature less than $\kappa = \frac{3\varepsilon}{2r_0(1-\varepsilon)^2}$ in $\Hs$. A standard fact about hyperbolic space is that any smooth curve with geodesic curvature bounded above by $1$ is embedded (see, for example, \cite[Lemma 3.6]{BBvariation}). It follows that by choosing $\varepsilon,r_0$ such that $\kappa \leq 1$ then $\Ep_{\O}(x_i) \neq \Ep_{\O}(y_i)$, a contradiction.
\end{proof}

\subsection{Explicit expression of Epstein maps in the upper-space model}\label{sec:explicit}
For concreteness, we also mention that in the upper-space model, Epstein maps have explicit expressions derived in \cite{epstein-envelopes,KrasnovSchlenker_CMP}.
We collect them here for the readers' convenience. 
We choose to include the simple derivations or examples to be specific about our conventions which is slightly different from \cite{KrasnovSchlenker_CMP} as we mentioned before.

Let $\rho = e^\varphi |\dd z|^2$ be a smooth conformal metric on an open set $U \subset \m C$. The Epstein map  
$\Ep_\rho : z \in U \mapsto (Z, \xi) \in \m C \times \m R_+ = \m H^3$  is given explicitly by
\begin{equation}\label{eq:Epsteincoordinates}
    \xi = \frac{2e^{-\varphi/2}}{1+|\varphi_\zbar|^2e^{-\varphi}},\qquad Z = z + \frac{2\varphi_{\zbar} e^{-\varphi}}{1+|\varphi_{\zbar}|^2e^{-\varphi}} =  z + \xi \cdot \psi,
\end{equation}
where 
$$\psi := \varphi_{\bar z} e^{-\varphi/2}, \qquad \varphi_\zbar = \partial_{\bar z} \varphi. $$

The Epstein Gauss map is $\widetilde \Ep_\rho \colon U \subset \m C \to T_1 \m H^3$ such that the base point is $\Ep_\rho$ and the vector component is  $\xi \normal$ where
\begin{equation}\label{eq:normalvector}
    \normal = \left( \frac{2\varphi_{\bar z} e^{-\varphi/2}}{1+|\varphi_{\zbar}|^2e^{-\varphi}} , \frac{1-|\varphi_{\zbar}|^2e^{-\varphi}}{1+|\varphi_{\zbar}|^2e^{-\varphi}} \right) = \left(\frac{2\psi}{1 +|\psi|^2}, \frac{1-|\psi|^2}{1+|\psi|^2}\right)
\end{equation}
is a Euclidean normal vector.
It is straightforward to check that the geodesic flow $\mf g_t (\widetilde \Ep_\rho (z)) \in T_1 \m H^3$ 
satisfies
$$\mf g_{-t} (\widetilde \Ep_\rho  (z)) = \widetilde \Ep_{e^{2t} \rho}  (z), $$
and the base point of $\mf g_{-t} (\widetilde \Ep_\rho (z))$ tends to $z$ as $t \to \infty$.

\begin{ex} \label{ex:epstein_maps}
\begin{itemize}
    \item If $\varphi \equiv 2 t$, then for all $z$,
    $$\Ep_\varphi (z) = (z, 2 e^{-t}) \qquad \normal = (0,1).$$
    \item If $e^\varphi  = \frac{4}{(1+ |z|^2)^2}$, then for all $z \in \m C$, $(Z, \xi) = (0, 1)$.
    \item If $\varphi = \log 4 - 2 \log (1-|z|^2)$, i.e., $\rho = e^\varphi |\dd z|^2$ is the hyperbolic metric in $\m D$, then for $z = r e^{\ii \t} \in \m D$,
    $$\Ep_\rho (r e^{\ii \t}) = \left(\frac{2r}{1+r^2} e^{\ii \t}, \frac{1-r^2}{1+r^2}\right) = \normal.$$
    This is consistent with Lemma~\ref{lem:example_disk} (and one of the advantages of choosing this convention is) that the Epstein--Poincar\'e map $\Ep_{\rho}$ maps $\m D$ onto the totally geodesic plane in $\m H^3$ bounded by $\partial \m D$. 
\end{itemize}
\end{ex}

More generally, we have the following explicit formula for the Epstein--Poincar\'e map associated with a simply connected domain $\O$. 
\begin{lemma}\label{lem:explicit_Poincare}
  Let $f : \m D \to \O$ be a conformal map and $\Ep_{\O} : \O \to \S_\O$ the  Epstein--Poincar\'e map. If $z = f(\z)$, $\z \in \m D$, then
\begin{equation*}
\begin{split}
\psi (z) & =  \varphi_{\zbar}e^{-\varphi/2}=  \frac{|f'(\zeta)|}{\overline{f'(\zeta)}}\left( -\frac{\overline{f''(\zeta)}}{\overline{f'(\zeta)}}\frac{(1-|\zeta|^2)}{2} + \zeta\right),\\
e^{-\varphi (z)/2} & =  \frac 12 |f'(\zeta)|(1-|\zeta|^2)\\
  \xi(z) & = \frac{2e^{-\varphi/2}}{1+|\psi|^2} = \frac{|f'(\zeta)|(1-|\zeta|^2)}{1+\big| -\frac{f''(\zeta)}{f'(\zeta)}\frac{(1-|\zeta|^2)}{2} + \overline{\zeta} \big|^2},\\
    Z (z) &=   z + \xi \cdot \psi = f(\zeta) + \frac{\left( -\frac{\overline{f''(\zeta)}}{\overline{f'(\zeta)}}\frac{(1-|\zeta|^2)}{2} + \zeta\right)f'(\zeta)(1-|\zeta|^2)}{1+\big| -\frac{f''(\zeta)}{f'(\zeta)}\frac{(1-|\zeta|^2)}{2} + \overline{\zeta} \big|^2}.
\end{split}
\end{equation*}

\end{lemma}

We verify from the explicit formulas that the expression of $\Ep_\O(f(0))$ coincides with the one in \eqref{eq:xi_0_EP} and that the Epstein map $\Ep_{\O}$ extends continuously to $\g$ as the identity map using the bounds
\begin{equation}\label{it:preS_bound}
\abs{\frac{(1-|\z|^2)}{2}\frac{f''(\z)}{f'(\z)} -  \bar \z} \le 2,  \qquad (1-|\z|^2)\abs{f'(\z)} \le 4\dist(f(\z), \g).
\end{equation}
See \cite[Prop.\,1.2, Cor.\,1.4]{Pommerenke_boundary}.

\section{Renormalized volume for a Jordan curve} \label{sec:V_R}

In this section, let $\g \subset \Chat$ be a Jordan curve and $\O$, $\O^*$ be the connected component of $\Chat \smallsetminus \g$. Let $\Ep_\O$ (resp. $\Ep_{\O^*}$) be the Epstein--Poincar\'e map associated with $\O$ (resp, $\O^*$). We write as before $\S_\O$ and $\S_{\O^*}$ for their images. 

\subsection{Disjoint Epstein--Poincar\'e surfaces}

When $\g$ is a circle,  both Epstein surfaces coincide with the geodesic plane bounded by $\g$ by Lemma~\ref{lem:example_disk}. We now show that in all other cases, the two Epstein surfaces of a Jordan curve are disjoint. We need the following special case of Grunsky's inequality, see, e.g., \cite[Thm.\,4.1, (21)]{Pom_uni} for the proof.
\begin{lemma}[Consequence of Grunsky inequality]\label{lem:Grunsky_inequality}
  Suppose that $f: \m D \to  \m C$ and $g : \m D^* \to  \Chat$ are univalent functions on $\m D$ and $\m D^*$ such that $f(0) = 0$ and $g(\infty) = \infty$, and $f(\m D) \cap g(\m D^*) = \emptyset$. Then 
  $$\int_{\m D} \abs{ \frac{f'(z)}{f (z)}  - \frac{1}{z} }^2  \dd^2 z +  \int_{\m D^*} \abs{\frac{g'(z)}{g(z)}  -\frac{1}{z} }^2  \dd^2 z  \le 2 \pi \log \abs{\frac{ g'(\infty)}{f'(0)}} $$
  where $g'(\infty) = \lim_{z \to \infty} g'(z)$. 
  Equality holds if  $\m C \smallsetminus \{f (\m D) \cup g (\m D^*)\}$ has zero Lebesgue measure.
\end{lemma}

Applying this, we obtain the following result.
\begin{prop}\label{prop:disjoint}
If $\g$ is not a circle, then $\S_\O$ and $\S_{\O^*}$ are disjoint.
\end{prop}
\begin{proof}
We consider $f:\m D\rightarrow \Omega$ and $g:\m D^* \rightarrow \Omega^*$ univalent maps. We  show that for $z \in \m D, w\in \m D^*$ the horospheres at $f(z), g(w)$ associated with the metrics $\rho_\O$ and $\rho_{\O^*}$ respectively are disjoint. By pre-composition and post-composition by  M\"obius transformations we can assume $z =0, w= \infty$ and $f(0) = 0, g(\infty) = \infty$. 
By Remark~\ref{rem:E_radius},
the Euclidean diameter
of the horosphere at $f(0) = 0$ is $|f'(0)|$ and the horosphere at $g(\infty) = \infty$ is the plane at Euclidean height $|g'(\infty)|$. This can be seen by considering the map $z \mapsto 1/g(1/z)$), which has the derivative at zero equaling $1/g'(\infty)$. Thus they are disjoint if  $|g'(\infty)| > |f'(0)|$. This follows from Lemma~\ref{lem:Grunsky_inequality} above unless $f'(z)/f(z) = 1/z$ and $g'(z)/g(z) = 1/z$.
But this implies both $f,g$ have Schwarzian zero, and therefore are M\"obius transformations with $\g$ is a circle. This contradicts the assumption.
\end{proof}

\subsection{Volume between the Epstein--Poincar\'e surfaces} \label{sec:vol_between}

 Let $\g \subset \Chat$ be an asymptotically conformal Jordan curve. 
 We now define the volume between $\S_\O$ and $\S_{\O^*}$. 
 Without loss of generality, we assume that $\g$ does not contain $\infty \in \Chat$ and use the upper half-space model of $\m H^3$.
  We cautiously note that both Epstein--Poincar\'e surfaces are non-compact and may not be embedded everywhere. For this reason, we use an approximation to compute the volume. For $\vare > 0$, let 
 $$\vol_\vare = \eta\left(\xi/\vare\right) \vol_{\mathbb{H}^3}$$
 where $\vol_{\mathbb{H}^3}$ is the hyperbolic volume form of the upper half-space model and $\eta:\mathbb{R}\rightarrow\mathbb{R}$ is a non-negative smooth function with $\supp(\eta)=\{\xi\geq1\}$ and $\eta^{-1}(1) = \{\xi\geq2\}$.
 
 Let $\varphi_\g$ be a continuous map $\overline{\m H^3} \to \overline {\m H^3}$, such that $\varphi_\g|_{\O} = \Ep_\O$, $\varphi_\g|_{\O^*} = \Ep_{\O^*}$, $\varphi_\g|_{\m H^3}$ is differentiable taking values in $\m H^3$ and with nonnegative Jacobian near $\g$.
 
 \begin{remark}\label{rem:extension}
 We give an example of such map $\varphi_\g$. Recall that for every quasicircle $\g$, there is a quasiconformal map $w:\Chat\rightarrow\Chat$ sending the unit circle to $\gamma$, so that $w$ is \emph{real analytic} in the complement of the unit circle (find for instance $w$ solution to the Beltrami equation $\partial_{\overline z}w = \mu \partial_z w$ for $\mu$ real analytic in the complement of the unit circle following \cite[Section III.1.1]{Lehto87}).
 Moreover, take a smooth family of Beltrami coefficients $(\mu_\xi)_{\xi> 0}$ defined as the convolutions $\mu\ast\varsigma_\xi$, where $\varsigma$ is a non-negative bump function around $0$ with total integral $1$ and $\varsigma_\xi =\xi^{-2}\varsigma\left(z/\xi\right)$. We then consider an extension of $w$ to a homeomorphism $w:\overline{\mathbb{H}^3}\rightarrow\overline{\mathbb{H}^3}$ defined leaf-wise from $\mathbb{C}\times\{\xi\}$ to $\mathbb{C}\times\{\xi\}$ as the solution of $\partial_{\overline z}w = \mu_\xi \partial_z w$ with the same images for $0, 1, \infty$ as our chosen $w:\Chat\rightarrow\Chat$. 
 By construction, $w$ sends $\D, \D^*$ and $S^1$ to $\O, \O^*$ and $\gamma$, respectively, and it is a diffeomorphism from $\mathbb{H}^3$ to itself. Denoting by $\operatorname{Arc}(\zeta)$ the hyperbolic geodesic with endpoints $\zeta, 1/\overline{\zeta}$ for $\zeta\in\D$, define a map $\varphi: \overline{\m H^3} \to \overline {\m H^3}$ that for $z\in\O$ sends $w(\operatorname{Arc}(w^{-1}(z))$  to the hyperbolic geodesic between $\Ep_\O(z)$ and $\Ep_{\O^*}(w (1/\overline{w^{-1} (z)} ))$. To obtain a map $\varphi$ which is differentiable away from $\gamma$, one can consider the Euclidean parametrization of arc length for $w(\operatorname{Arc}(w^{-1}(z))$ and send it to the geodesic between $\Ep_\O(z)$ and $\Ep_{\O^*}(w (1/\overline{w^{-1} (z)} ))$ with constant speed with respect to the hyperbolic arc length. 
  Since both Epstein surfaces are disjoint (unless $\g$ is a circle) by Proposition~\ref{prop:disjoint} and embedded near the boundary by Corollary~\ref{cor:embedding_wp}, we check easily that the Jacobian of $\varphi$  is positive in a neighborhood $U_\g$ of $\g$ in $\m H^3$. If $\g$ is a circle, the Jacobian is zero.
   
 \end{remark}
 
 We define
 $$V_2(\g) (\vare) := \int_{\m H^3} \varphi_\g^* \vol_\vare.$$
 This is the \emph{signed volume between the Epstein surfaces} bounded by $\g$ and weighted by $\eta$ at scale $\vare$.
 
 Since the boundary values of $\varphi_\g$ are determined and $\varphi_\g (\overline{\m H^3}) \cap \{(Z,\xi) \,|\, \xi \ge \vare\}$ is compact, $V_2(\g) (\vare)$ is finite and independent of the choice of $\varphi_\g$. Indeed, given two maps $\varphi^0_\g$, $\varphi^1_\g$, the convex combination $\varphi^t_\g:= (1-t)\varphi^0_\g + t\varphi^0_\g$ defines a homotopy of maps satisfying the same conditions. Take $N\subseteq \overline{\mathbb{H}^3}$ smooth submanifold so that it contains the support of $(\varphi^t_\g)^*(\vol_\vare)$ for any $0\leq t\leq1$. By Stokes theorem, it follows $\int_{\partial (N\times [0,1])} i^*((\varphi^t_\g)^*\vol_\vare)=0$, where $i:\partial (N\times [0,1])\rightarrow N\times [0,1]$ is the inclusion map. Since $\varphi^t_\g$ is point-wise constant in $\partial N \cap \Chat$ along $0\leq t\leq1$ and in $\partial N \smallsetminus \Chat$ does not hit the support of $\vol_\vare$, 
this implies $\int_N (\varphi^1_\g )^*\vol_\vare - \int_N (\varphi^0_\g) ^*\vol_\vare = 0$.

We claim that the limit
\begin{equation}\label{eq:limit_V}
 V (\g) : = \lim_{\vare \to 0\splus} V_2(\g) (\vare) \in (-\infty,\infty]
 \end{equation}
 exists. To see this, we note first that 
 $\int_{U_\g} \varphi_\g^* \vol_\vare$ increases as $\vare \to 0\splus$, where $U_\g$ is a neighborhood of $\g$ where the Jacobian of $\varphi_\g$ is nonnegative. 
 Since  
 $\overline{\m H^3} \smallsetminus U_\g$ is 
 a compact set which does not intersect with $\g$, the image by $\varphi_\g$ is a compact subset of $\m H^3$.  Therefore, there exists $\vare_0>0$ such that 
$$\varphi_\g (\overline{\m H^3} \smallsetminus U_\g) \subseteq \{(Z,\xi) \colon \xi \ge \vare_0\}.$$
 Thus $\int_{\m H^3 \smallsetminus U_\g} \varphi_\g^* \vol_\vare$ is constant for small enough $\vare$ and $\lim_{\vare \to 0\splus} V_2(\g)(\vare)$ exists. The monotonicity and \eqref{eq:Eps_invariant} also show that the limit is invariant under the action of elements in $\pslt$ that do not send any point of $\g$ to $\infty \in \Chat$. Indeed, for any such $M\in \pslt$ and $\varepsilon>0$, $M\circ\varphi_\g\circ M^{-1}$ satisfies the conditions for a candidate for $\varphi_{M(\g)}$ while there exists $\varepsilon'>0$ so that $$M\left(\varphi_\g(\mathbb{H}^3)\cap \{(Z,\xi) \colon \xi \ge \vare\}\right) \subseteq (M\circ\varphi_\g\circ M^{-1})(\mathbb{H}^3) \bigcap \{(Z,\xi) \colon \xi \ge 2\vare'\}.$$
 Taking $\vare$ small enough so that $\overline{\m H^3}\smallsetminus M\left(\varphi_\g(\mathbb{H}^3)\cap \{(Z,\xi) \colon \xi \ge \vare\}\right)$ is in the region around $M(\gamma)$ where Corollary~\ref{cor:embedding_wp} applies and where we have that $\varphi_\g$ has positive Jacobian as in Remark~\ref{rem:extension}, we use $M^*\vol_{\mathbb{H}^3}=\vol_{\mathbb{H}^3}$ to get $\varphi_\g ^*\vol_\varepsilon \leq (M^{-1})^*(M\circ\varphi_\g\circ M^{-1})^* \vol_{\vare'} $. Hence we get 
 $$V_2(\g)(\vare) = \int_{\mathbb{H}^3}\varphi_g^*\vol_\vare \leq \int_{\mathbb{H}^3} (M^{-1})^*(M\circ\varphi_\g\circ M^{-1})^* \vol_{\vare'} = V_2(M(\g))(\vare'),$$ 
 from which it follows that $V(\g)\leq V(M(\g))$. Since we get the reverse inequality by considering $M^{-1}$, $V(\g) = V(M(\g))$ follows.

Summarizing the above, we obtain the following definition.
\begin{df}\label{def:volume}
   For an asymptotically conformal Jordan curve $\g \subset \Chat$, we define \emph{the signed volume between the Epstein--Poincar\'e surfaces} $V(\g)$ to be the limit in \eqref{eq:limit_V} applied to the curve $A (\g)$, where $A$ is any element in $\pslt$ such that $A (\g)$ does not pass through $\infty$.
\end{df}

\subsection{Volume for smooth Jordan curves}

In this subsection, we will see if the Jordan curve $\gamma$ is sufficiently smooth, then the map $\Ep_\Omega$ extends not only continuously to $\gamma$ but also osculates to the totally geodesic plane bounded by the circle osculating to $\gamma$ in $\m C$. This will be useful later to prove that the volume between $\Ep_\O$ and $\Ep_{\O^*}$ is finite if $\gamma$ is sufficiently smooth.

If $\g$ is $C^{4,\a}$  for some $0 <\a <1$, Kellogg's theorem  (see, e.g., \cite[Thm.\,II.4.3]{GM}) implies that the conformal map $f : \m D \to \O$ extends to a $C^{4,\a}$ homeomorphism  $\overline{\m D} \to \overline \O$. Hence by Lemma \ref{lem:explicit_Poincare} the Epstein map $\Ep_\O$ extends to $\partial \O$ as a $C^{2,\alpha}$ map. We define the {\em osculating circle} $\calC_\gamma(\t)$  at $\gamma(e^{\ii \t})$ to be the circle tangent to $\g$ at $\gamma(e^{\ii \t})$ and with the same curvature as $\gamma$ at $\gamma(e^{\ii \t})$. We then define the \emph{osculating plane}  $\calP_\gamma(\t)$ to be  the geodesic plane in $\mathbb{H}^3$ so that the boundary of $\calP_\gamma(\t)$ is $\calC_\gamma(\t)$.

\begin{prop}\label{prop:secondorderplane}
Let $\gamma$ be a $C^{4,\alpha}$ Jordan curve in $\mathbb{C}$ for some $0 <\a <1$. Then $\Ep_\O$ and $\calP_\gamma(\t)$, viewed as surfaces in $\m R^3$,  are tangent at $(\gamma(e^{\ii \t}),0)$ and agree up to second order.
\end{prop}
\begin{proof}

We pick a point $z_0 = x_0 + \ii y_0 \in \g = \partial \O$. By post-composing by an M\"obius transformation, we can assume that $z_0 = 0$.  Let $f: \m H \rightarrow \Omega$ be a conformal map, where $\m H = \{ \z = u + \ii v \colon v > 0\}$ denotes the half-plane. Without loss of generality, we assume that  $f(0) = 0$, $f'(0) = 1$  and $f''(0) = 0$. Thus the plane bounded by the osculating circle (which is the line $\{y = 0\} \subset \m C$) at $z_0$ is the plane $\{y = 0\} \subset \m C\times \m \R_+ = \{x +\ii y + \jj \xi \colon \xi > 0\}$. We now show that  
$\Ep_\Omega(f(\z)) = u+\jj v+ O(|\z|^3)$ for $|\z|$ small, which would imply
$\Ep_\Omega(x + \ii y) = x+\jj y+ O(|x + \ii y|^3)$ and thus
the statement would follow.

We take $\z = u+\ii v \in \m H$. We let $g_\z:\m D \rightarrow \Omega$ uniformize $\Omega$ with $g_\z(0) = f(\z)$.
The osculating M\"obius transformation $M_\z$ for $g_\z$ at $0$ is  
$$M_\z(w) = f(\z)+\frac{\a w}{\beta w + 1/\a}$$
with $\a^2 = g_\z'(0)$ and $2\a\b = -g_\z''(0)/g_\z'(0)$. Then by Lemma \ref{lemma:osc-mob}  
$$\Ep_\Omega(f(\z))  = M_\z(\jj) = f(\z) + \frac{\alpha\overline{\beta} + \jj}{|\beta|^2+1/|\alpha|^2} = f(\z) + \frac{\alpha^2(\overline{\alpha\beta}) + |\alpha|^2\jj}{|\alpha\beta|^2+1}.$$
We  choose $g_\z = f \circ \varphi_\z$ where $\varphi_\z : \m D \to \m H$ is given by 
$$\varphi_\z(w) =u+ \ii v\left(\frac{1-w}{1+w}\right).$$
Thus as $\varphi_\z (0) = \z$, $\varphi_\z'(0) = -2\ii v, \varphi_\z''(0)= 4\ii v$, 
\begin{eqnarray*}
\a^2 &=& g_\z'(0) = f'(\z)\varphi_\z'(0) = -2\ii v f'(\z) = -2\ii v+ O(|\z|^3)\\
 g_\z''(0) &=&  f'(\z)\varphi_\z''(0) + f''(\z)\varphi'_\z(0)^2= 4\ii v(f'(\z)+\ii vf''(\z)) = 4\ii v +O(|\z|^3)\\
 \a\b &=& -\frac{g_\z''(0)}{2g_\z'(0)} = 1+\ii v\frac{f''(\z)}{f'(\z)}= 1+O(|\z|^2).
 \end{eqnarray*}
As $f(\z) = \z + O(|\z|^3)$ then a straightforward computation shows 
$$\Ep_\Omega(f(\z)) = \z+\frac{1}{2}(-2\ii v+2\jj v) + O(|\z|^3) = u+\jj v+ O(|\z|^3)$$
as claimed.
\end{proof}

 Next, we show that for sufficiently regular curves $\gamma$, this volume is, in fact, finite.

\begin{prop}\label{prop:finitevolume}
     Let $\gamma$ be a $C^{5,\alpha}$ Jordan curve in $\mathbb{C}$ for some $0 <\a <1$. Then $V(\gamma)$ is finite.
\end{prop}
\begin{proof}
Without loss of generality, we assume that $\gamma$ is parametrized by arc-length 
 as a function $\m S^1\rightarrow\mathbb{C}$. Take $\varphi_\g$ some continuous map $\overline{\m H^3} \to \overline {\m H^3}$ as before, \ie $\varphi_\g|_{\O} = \Ep_\O$, $\varphi_\g|_{\O^*} = \Ep_{\O^*}$, and $\varphi_\g|_{\m H^3}$ is differentiable. We take the following $C^{4,\alpha}$ parametrization of a neighborhood $U$ of $\gamma$ in $\overline{\mathbb{H}^3}$, denoted $G:\m S^1_s\times \overline{\mathbb{H}^2}_{(a,b)} \rightarrow \overline{\mathbb{H}^3}$, 
by
\begin{equation}\label{eq:G_coordinate}
G(s,a,b) = \gamma(s) + \ii a\gamma'(s) + \jj b.
\end{equation}
It is a straightforward calculation to see that the hyperbolic metric in $G$-coordinates is given by 
$$\frac{(1-ak(s))^2}{b^2}\dd s^2 + \frac{1}{b^2}\dd a^2 + \frac{1}{b^2}\dd b^2,$$
where $k(s)$ is the signed curvature of $\gamma$ given by $\gamma''(s)=\ii k(s)\gamma'(s)$. Hence the volume form is given by 
$$\frac{(1-ak(s))}{b^3}\, \dd s \,\dd a\, \dd b.$$

If we assume that $\gamma$ is $C^{5,\alpha}$, then the Epstein--Poincar\'e surfaces are $C^{3,\a}$ up to the boundary. This means that there are $C^{3,\alpha}$ functions $a_\Omega,a_{\Omega^*}:\m S^1_s\times [0,\vare_0]_b\rightarrow\mathbb{R}$ so that the Epstein--Poincar\'e surfaces in the neighborhood $U$ of $\gamma$ are parametrized by $G(s,a_\Omega(s,b),b)$, $G(s,a_{\Omega^*}(s,b),b)$. And since by Proposition \ref{prop:secondorderplane} the Epstein--Poincar\'e surfaces agree up to second order at $\gamma$, then there exists a constant $C>0$ so that $|a_\Omega(s,b)-a_{\Omega^*}(s,b)|\leq Cb^3$.

Hence for small enough neighborhood $U$ of $\g$, the integral $\int_{U}\varphi_\g^* \vol$ will be bounded by

\[V_1(\g)(\vare_0) := \int_{\m S^1}\int_0^{\vare_0} \int_{a_{\Omega}(s,z)}^{a_{\Omega^*}(s,z)}\frac{(1-ak(s))}{b^3} \, \dd a \, \dd b\, \dd s.
\]

This integral is well-defined and convergent since 
\begin{align*}
  &\bigg\vert\int_{a_{\Omega}(s,b)}^{a_{\Omega^*}(s,b)}\frac{(1-ak(s))}{b^3}\dd a\bigg\vert \\
  & = \frac{1}{b^3}\bigg\vert a_{\Omega^*}(s,b)-a_\Omega(s,b)\bigg\vert.\bigg\vert\bigg(\frac{a_{\Omega}(s,b)+a_{\Omega^*}(s,b)}{2}\bigg)k(s)-1\bigg\vert  
\end{align*}
is bounded by a constant independent of $(s,b)$. Hence 
$$V(\g)= \lim_{\vare \to 0\splus} V_2(\g)(\vare)$$ is finite. 
\end{proof}

\begin{df}\label{df:RenormalizedVolume}
Let $\gamma$ be a Weil--Petersson quasicircle in $\mathbb{C}$. Then we define $V_R(\gamma)$, the renormalized volume of $\gamma$, as
\begin{align*}
\begin{split}
    & V_R(\gamma) :=   V(\gamma) - \frac{1}{2} \int_{\S_{\O} \cup \S_{\O^*}} H \dd a\\
    & = V(\gamma) - \frac12 \left(\int_{\O} \norm{\mc S(f^{-1})}^2 (z) \rho_{\O} (z) \dd^2 z   + \int_{\O^*} \norm{\mc S(g^{-1})}^2 (z) \rho_{\O^*} (z) \dd^2 z\right).
\end{split}
\end{align*}
\end{df}

\begin{remark}\label{remark:VRdefinition}
The second identity follows from Theorem~\ref{thm:PE}. A priori, $V_R(\g) \in (-\infty, \infty]$ as $V(\g) \in (-\infty, \infty]$ and the integrals of mean curvature are finite by Corollary~\ref{cor:WP_HdA}.
Proposition~\ref{prop:finitevolume} shows that if $\g$ is $C^{5,\a}$, then $V_R(\g) < \infty$.
From the $\pslt$-invariance of each summand in Definition \ref{df:RenormalizedVolume}, we can easily see that $V_R$ is $\pslt$-invariant.
\end{remark}

\section{Universal Liouville action as renormalized volume}\label{section:LoewnerVR}

Our objective in this section is to prove that the renormalized volume in Definition \ref{df:RenormalizedVolume} agrees up to a constant with the Loewner energy for $C^{5,\alpha}$ curves. 

\subsection{Schl\"afli formula for the variation of the volume}

We say that $(\gamma_t)_{t \in [0,1]}$ is a $C^{k,\a}$ family of Jordan curves in $\m C$ for some $k \ge 1$, $0 < \a <1$, if $\g_t$ is $C^{k,\a}$ for all $t \in [0,1]$, and we can choose the conformal maps $f_t : \m D \to \O_t$ to be jointly $C^{k,\a}$ for $(t, z) \in [0, 1] \times \overline{\m D}$ and $g_t : \m D^* \to \O_t^*$ to be $C^{k,\a}$ for $(t, z) \in [0, 1] \times \overline{\m D^*}$.
\begin{ex} \label{ex:deformation}
      Let $\g = \g_1$ be a  $C^{k,\a}$ Jordan curve in $\m C$ for some $k \ge 1$, $0 < \a <1$, and $D$ be a small round disk in $\O$ and $A = \O \smallsetminus D$. Then there is $0< r < 1$ such that the round annulus $A_r = \{z \in \m C \,|\, r< |z|< 1\}$ is conformally equivalent to $A$ (sending $\m S^1$ onto $\g$ and $r \m S^1$ onto $\partial D$). It is easy to see using Kellogg's theorem that any conformal map $F : A_r \to A$ is $C^{k,\a}$ on $\overline {A_r}$. The family of Jordan curves $(\g_t = F ((r + t(1-r)) \m S^1))_{t \in [0,1]}$ is a $C^{k,\a}$ family of Jordan curves. This can be seen using \cite{Roth_Schippers}.
\end{ex}

For simplicity, for a quantity $Q$ which is a differentiable function of a parameter $t$ on an open interval about $0$, we let $\delta Q(t)$ be the derivative at $t$ and $\delta Q$ be its derivative at $0$, i.e.
 $$\delta Q = \frac{\partial}{\partial t}\bigg|_{t=0}Q(t).$$

\begin{thm}\label{thm:Schl\"afliformula}
Let $(\gamma_t)_{t \in [0,1]}$ be a $C^{5,\alpha}$ family of Jordan curves for some $0 < \alpha <1$. Then the first derivative of the volume $\delta V(\gamma) = (\partial/\partial_t)|_{t = 0} V(\g_t)$ is computed by
\[
\delta V(\gamma) = \int_{\Omega} \Ep_{\Omega}^*\left(\delta H + \frac14\brac{ \delta\I, \II } \,\dd a\right) + \int_{\Omega^*}  \Ep_{\Omega^*}^*\left(\delta H + \frac14\brac{ \delta\I, \II } \,\dd a\right)
\]
where $\Ep_{\Omega}, \Ep_{\Omega^*}$ are the Epstein--Poincar\'e maps of $\Omega, \Omega^*$ (respectively); $\I, \II, H, \dd a$ are the metric, second fundamental form, mean curvature and area form on $\S_\O$ and $\S_{\O^*}$ associated with $\g_0$.
\end{thm}

The proof of this theorem when the Epstein--Poincar\'e surfaces are immersions everywhere is completed in Theorem~\ref{thm:schlafli_imm}. We postpone the proof of the non-immersed case to Section~\ref{gen-Schl\"afli}. 

To prove this, we will decompose the region between the  Epstein surfaces into two subregions and analyze the variation on both. This will require us to use a Schl\"afli formula --- the key ingredient --- for the variation of regions with piecewise smooth boundaries.

We consider $F: \overline{\m B} \rightarrow \m H^3$ a parametrization of a region $R$ by the closed unit $3$-ball $\overline{\m B}$ such that $F$ is an immersion in $\m B$ and the boundary map  $E: \hat{\m C}\rightarrow \m H^3$ is piecewise smooth. We let $V$ be the hyperbolic volume of $R$ defined as $\int_{\m B} F^* \vol_{hyp}$. For our purposes, we can assume that $E$ is piecewise smooth on two disjoint smooth simply connected domains $\Omega_1,\Omega_2$ and the annulus $A$ between them. Further, wherever $E$ is an immersion, we let $B$ be the pullback of the shape operator by $E$. Now we consider a smooth variation of $E$ by maps $F_t, E_t$ with the same decomposition of $\hat{\m C}$ and let $\xi = \partial_t  E_t|_{t = 0}$ be the vector field on $\partial R$ describing this variation.

We consider the following version of Schl\"afli.
\begin{thm}[See {\cite[Thm\,4]{Souam04}}]{\label{theorem:Schl\"afli_pw}}
Let $F_t, E_t$ be variations of $F,E$ as above such that $E$ is an immersion on $\Omega_1,\Omega_2$, and $A$. Then the variation of volume satisfies
\begin{align*}
2 \delta V = \int_{\Omega_1\cup\Omega_2\cup A} \hspace{-10pt}\tr\brac{ \nabla_\xi (B\cdot),D E\cdot} E^*\dd a 
 +
\left(\int_{\partial\Omega_1}  
\delta \theta_1 
E^*\dd\ell_1 +\int_{\partial\Omega_2} 
\delta \theta_2 E^*\dd\ell_2\right)
\end{align*}
where $\tr$ denotes the 
trace of a $2$-tensor with respect to the induced metric, $\dd a$ denotes the induced area form, $\theta_i$ is the exterior dihedral angle between the regions $\Omega_i, A$, and $\dd\ell_i$ is the length measure of $\partial\Omega_i$. 
\end{thm}

The referenced variational formula in \cite{Souam04} differs from the above in the first term on the right-hand side. We obtain our formula above using the following lemma.  

\begin{lemma}[{See \cite[Eq.\,(3.1) and Prop.\,5]{Souam04}}] \label{lem:Tr_variation_V}
    At points where $E$ is an immersion, the form $\frac12\tr\brac{ \nabla_\xi (B\cdot),DE \cdot} E^*\dd a$ agrees with the pullback by $E$ of the form $(\d H + \frac14\brac{ \d \I, \II })\,\dd a$.
\end{lemma}

 We now describe the decomposition for $C^{5,\alpha}$ curves.

\subsection{Decomposition} \label{sec:Decomposition}
We consider a $C^{5,\alpha}$ family $(\gamma_t)_{t \in [0,1]}$ of Jordan curves ($\alpha>0$). 
We will define a parametrization of the Epstein surfaces that allows us to compute the derivative $\frac{\partial}{\partial t}
V(\gamma_t)$. Since scalar multiplications are isometries of $\mathbb{H}^3$, we can assume without loss of generality that all curves $\gamma_t$ have Euclidean arclength $2\pi$. Furthermore, for  $\vare$ small, $V(\gamma_t) = V_1(\gamma_t)(\vare) + V_2(\gamma_t)(\vare)$, where $V_2$ is defined in Section~\ref{sec:vol_between}. Moreover, we also know that $V_2(\gamma_t)(\vare) \xrightarrow[]{\vare\rightarrow 0} V(\gamma_t)$ converges uniformly in $t$ by the proof of Proposition~\ref{prop:finitevolume}.

Let $f_t:\mathbb{D}\rightarrow\Omega_t$, $g_t:\mathbb{D}^*\rightarrow\Omega^*_t$ be univalent functions that extend to $C^{k,\a}$ functions on $[0, 1] \times \overline{\m D}$ and $[0, 1] \times \overline{\m D^*}$ respectively. 
Consider $\vare$ sufficiently small so that for $z\in\overline{\mathbb{D}}$ with $|z|>1-\vare$ then $\Ep_{\Omega_t}(f_t(z))$ belongs to the parametrized neighborhood $U_{\gamma_t}$ from Proposition \ref{prop:finitevolume} for all $t \in [0,1]$. 
Take the horizontal line $L_{t,z}$ (horocycle centered at $\infty \in \Chat$) obtained by varying the second $G$-coordinate \eqref{eq:G_coordinate} in $U_{\gamma_t}$ starting from $\Ep_{\Omega_t}(f_t(z))$, and define $h_t(z) \in \overline{\mathbb{D}^*}$ to be the point such that $\Ep_{\Omega^*_t}(g_t(h_t(z)))$ is the first point of intersection of the horizontal line with $\Ep_{\Omega^*_t}$. See Figure~\ref{fig:h_t} for an illustration.

Clearly along $\partial\mathbb{D}$ the map $h_t$ agrees with $g_t^{-1}\circ f_t$, and from the regularity of $f_t, g_t$ and the $G$-coordinates of $U_{\gamma_t}$ we can see that the 1-parameter family of functions $h_t$ is $C^{3,\alpha}$ in $1-\vare<|z|\leq 1$ and $t$-parameters. Moreover, by the inverse function theorem, we can take $\vare$ sufficiently small so that
$h_t$ is a diffeomorphism from $1-\vare<|z|\leq 1$ with its image in $\overline{\m D^*}$.

\begin{figure}
    \centering
    \includegraphics[width=\textwidth]{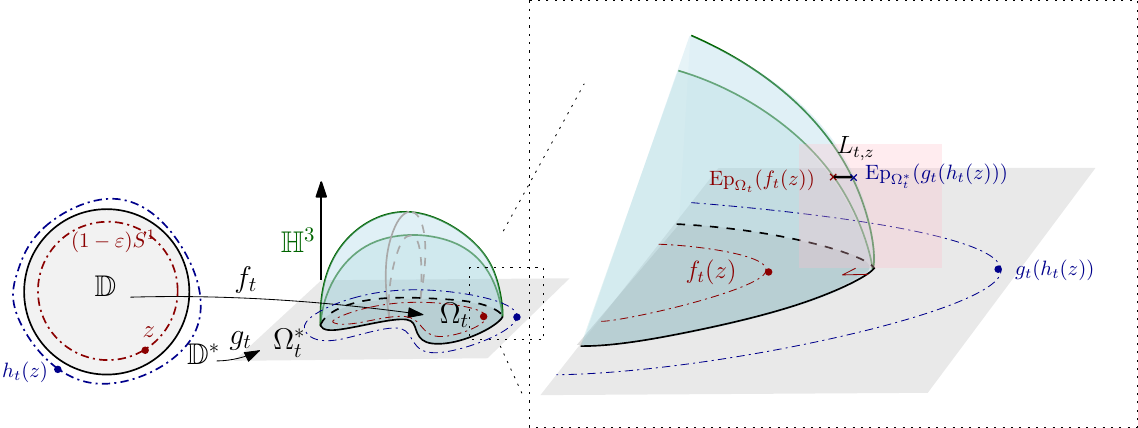}
    \caption{Illustration of the two Epstein--Poincar\'e surfaces associated with the two connected component of $\Chat \smallsetminus \g_t$ and the map $h_t$.}
    \label{fig:h_t}
\end{figure}

For $1-\vare < r < 1$, define the cylindrical neighborhood of $\gamma_0$ as
$$A(r) = f_0(\lbrace r\leq|z|\leq 1 \rbrace)\cup g_0(h_0(\lbrace r\leq|z|\leq 1 \rbrace)),$$ which we parametrize by $\m S^1\times[r,1/r]$, sending $(p,s)$ to $f_0(sp)$ if $s\leq 1$ and sending $(p,s)$ to $g_0(h_0(\frac{p}{s}))$ if $s\geq 1$.
These cylindrical neighborhoods are nested as $r$ grows, and their intersection as $r\rightarrow1^-$ is $\gamma_0$. Define as well $\Omega(r), \Omega^*(r)$ to be the components of $\mathbb{C}\smallsetminus A(r)$ in $\Omega_0$ and $\Omega^*_0$, respectively.

Define a 1-parameter family of homeomorphisms
$F_t:\Chat\rightarrow\Chat$ so that for $z\in\overline{\Omega_0}$ we define $F_t(z):=f_t(f_0^{-1}(z))$, for $z\in g_0(h_0(\lbrace 1-\vare < |z| \leq 1 \rbrace))$ we define 
$F_t (z) = g_t \circ h_t \circ h_0^{-1}\circ g_0^{-1} (z)$. In this way, if $\Ep_{\O_0} (u)$ and $\Ep_{\O_0^*} (v)$ are connected by the line $L_{0,z}$ (namely, $u = f_0(z)$ and $v = g_0 \circ h_0 (z)$), then $\Ep_{\O_t} (F_t(u))$ and $\Ep_{\O_t^*} (F_t(v))$ are connected by the line $L_{t,z}$.
We extend $F_t$ to the rest of $\Omega^*_0$ as a $C^{3,\alpha}$ map in both $\overline{\Omega^*_0}$ and $t$ parameters. Let us also fix $F_0$ to be the identity. It follows then that $F_t|_{\Omega_0}$ is a conformal map between $\Omega_0$ and $\Omega_t$, and $F_t|_{\gamma_0}$ parametrizes $\gamma_t$. For $1-\vare<r<1$, we construct the family of piecewise smooth maps $E_{r,t}:\Chat\rightarrow \mathbb{H}^3$ satisfying the following properties:

\begin{enumerate}[label={(C\arabic*)}]
    \item \label{item:rcutoff} In $\Omega(r), \Omega^*(r)$ the map $E_{r,t}$ is defined as the composition of  the Epstein--Poincar\'e maps $\Ep_{\Omega_t}, \Ep_{\Omega^*_t}$ with $F_t$.
    \item \label{item:horizontal step} Considering the parametrization of $A(r)$, for each $p\in \m S^1$ then $E_{r,t}(\lbrace p\rbrace\times[r,1/r])$ is the straight horizontal segment $L_{t,rp}$. 
\end{enumerate}

The map $E_{r,\cdot} (\cdot)$ is piecewise smooth, and it is $C^{3,\alpha}$ while restricted to $[0,1]\times\Omega(r), [0,1]\times A(r), [0,1]\times \Omega^*(r)$.

\begin{lemma}\label{lemma:piecewiseshape}
Along each $\Omega(r),A(r), \Omega^*(r)$, on the image of $E_{r,t}(p)$ there is a well-defined unit vector $\vec n$ that is normal to the image of $DE_{r,t}$. 
On $\Omega(r)$ and $\Omega^*(r)$, $\vec n$ coincides with $\widetilde \Ep_{\O_t}$ and $\widetilde \Ep_{\O^*_t}$ 
respectively and on $A(r)$, we choose $\vec n$ to have positive vertical component. 
The corresponding Euclidean unit vector $\vec \eta$ varies piecewise $C^{3,\alpha}$ on $\{(r,t, p) \,|\, r \in (1-\vare, 1], t \in [0,1], p \in \O(r) \text{ or } A(r) \text{ or } \O^*(r)\}$, and when $r = 1$, $p\in A(r) = \gamma$, $\vec \eta \equiv (0,0,1)$.
\end{lemma}

\begin{proof}
The regularity of $\vec n$ on $\Omega(r), \Omega^*(r)$ follows from the construction of the Epstein--Poincar\'e map, see \eqref{eq:normalvector}. 
The regularity of $\vec n$ in $A(r)$ for $ 1-\vare < r <1$ can be seen using the $G$-coordinates parametrizing $\S_{\O_t}$ and $\S_{\O_t^*}$.

To obtain the regularity of the Euclidean unit vector $\vec \eta$ on $A(r)$ up to $r = 1$ and its value $(0,0,1)$, we use the $G$-coordinates and the expression of $\xi$ in Lemma~\ref{lem:explicit_Poincare}.
\end{proof}

\begin{remark}
We will simplify notation by dropping $r,t$ sub-indices when appropriate.
\end{remark}

\subsection{Proof of Schl\"afli formula} 
We first prove Theorem~\ref{thm:Schl\"afliformula} under the added assumption that the Epstein--Poincar\'e surfaces are immersions.

\begin{thm}\label{thm:schlafli_imm}
Let $(\gamma_t)_{t \in [0,1]}$ be a $C^{5,\alpha}$ family of 
Jordan curves ($\alpha>0$) such that the Epstein--Poincar\'e surfaces are \emph{immersions}. Then the first derivative of the volume $V(\gamma)$ is computed by

\[\delta V(\gamma) = \int_{\Omega} \Ep_{\Omega}^*\left(\delta H + \frac14\brac{ \delta\I, \II } \dd a\right) + \int_{\Omega^*}  \Ep_{\Omega^*}^*\left(\delta H + \frac14\brac{ \delta\I, \II } \dd a\right).
\]
\end{thm}

\begin{remark}
  To remove the assumption of immersion and to prove Theorem~\ref{thm:Schl\"afliformula} will require some technical analysis which we leave to a later section (see Section~\ref{gen-Schl\"afli}). As by Theorem \ref{thm:PE} the map $\Ep_{\O}$ (respectively for $ \Ep_{\O^*}$) is an immersion in $\{z \in \O \ | \ \|\mc S(f^{-1})(z)\|_\O \neq 1\}$, in particular we will extend continuously the right-hand side of the formula in Theorem~\ref{thm:Schl\"afliformula} to the locus $\{z \in \O \ | \ \|\mc S(f^{-1})(z)\|_\O = 1\}$ as smooth differential forms so the variation of volume formula still holds.
\end{remark}

\begin{proof}
For $r$ close to $1$ we define $V_2(r,t)$ as the volume bounded by $E_{r,t}$. Similarly, we define $V_1(r,t)$ as the volume of the region between $\Sigma_\O, \Sigma_{\O^*}$ outside of $E_{r,t}$. Then  $V(\gamma_t) = V_1(r,t)+V_2(r,t)$.

We first show that
\begin{align*}
  \lim_{r\rightarrow 1^-} 
  \delta
  V_2(r,t) = & \int_{\Omega} \Ep_{\Omega}^*\left(\delta H + \frac14\brac{ \delta\I, \II } \dd a\right) \\
  &+ \int_{\Omega^*}  \Ep_{\Omega^*}^*\left(\delta H 
  + \frac14\brac{ \delta\I, \II } \dd a\right).  
\end{align*}
 Combining Theorem \ref{theorem:Schl\"afli_pw} and Lemma \ref{lem:Tr_variation_V},  we only need to prove that 
\[\lim_{r\rightarrow 1^-}\int_{A(r)} \frac12\tr\brac{ \nabla_\xi (B\cdot),DE_p\cdot} =0 
\]
and
\[\lim_{r\rightarrow 1^-} \frac12\bigg( \int_{\partial \Omega(r)} \frac{\partial \theta}{\partial t} E^*\dd \ell + \int_{\partial \Omega^*(r)} \frac{\partial \theta^*}{\partial t} E^*\dd \ell\bigg) = 0.
\]

For the first term, observe that $A(r)$ belongs to the surface described in \ref{item:horizontal step}. These families of surfaces can be described by
\begin{align*}
    \{p \in \m S^1,  s \in [r, 1/r], & r \in (1-\vare,1], t\in [0,1]\} \rightarrow \overline{\m H^3} \subset \m R^3 
\\
(p,s,r,t)& \mapsto (x(p,s,r,t),y(p,s,r,t),z(p,s,r,t)),
\end{align*}
where $p,s$ parametrize the surface as in \ref{item:horizontal step} for $\gamma_t$. This parametrization extends $C^{3,\alpha}$ for $r=1$ towards the boundary of $\mathbb{H}^3$ by making $z(p,s,1,t)\equiv 0$. Moreover, given \ref{item:rcutoff} and Lemma~\ref{lem:explicit_Poincare} then that $z(p,s,r,t)=O(1-r)$ uniformly for all other parameters.

Hence the first and second fundamental form (as well as their first order variations) are of order at most $(1-r)^{-2}$, and the inverse of the first fundamental form has order $(1-r)^2$. This follows from the expression of these fundamental forms in terms of the derivatives up to the third order of the parametrization into $\m R^3$ and the conformal factor $z(p,s,r,t)=O(1-r)$. Thus, the terms $H, \delta H, \brac{ \delta I, \II }$ are uniformly bounded. As by Proposition~\ref{prop:secondorderplane} $E_{r,t}(A(r))$ has euclidean area $O((1-r)^3)$ so
the hyperbolic area of $E_{r,t}(A(r))$ is of order $O(1-r)$. Therefore 
\[\lim_{r\rightarrow 1^-}\int_{A(r)} \frac12\tr\brac{ \nabla_\xi (B\cdot),DE_p\cdot}=0.
\]
Likewise, the exterior dihedral angle $\theta$ that takes each $(p, r,t)$ to the angle between $E_{r,t} (\Omega(r))$ and $E_{r,t}(A(r))$ at $(x(p,r, r,t),y(p,r,r,t),z(p,r,r,t)) = : \gamma_r (p)$, extends smoothly to $r=1$ as right angles. Similarly for $\theta^*$ the exterior dihedral angle between $E_{r,t} (\Omega^*(r))$ and $E_{r,t}(A(r))$ along $\gamma_r^*$, where $\gamma_r^* (p) : = (x(p,1/r, r,t),y(p,1/r,r,t),z(p,1/r,r,t)) $. 
We use again that the Epstein--Poincar\'e surfaces agree up to second order (Proposition~\ref{prop:secondorderplane}) to use parametrizations $\gamma_r(p)$ and $\gamma_r^*(p)$ satisfying
\begin{equation}
\begin{split}
    \bigg\Vert\frac{\partial \gamma_r}{\partial p}(p) - \frac{\partial \gamma^*_r}{\partial p}(p)\bigg\Vert&\leq C(1-r)^2\\
     \bigg|\frac{\partial \theta(p)}{\partial t} + \frac{\partial \theta^*(p)}{\partial t}\bigg|&\leq C(1-r)^2
\end{split}
\end{equation}
for some uniform constant $C>0$.
Then since the last coordinate of $\gamma_r$ and $\gamma_r^*$ is $O(1-r)$, for some uniform constant $C>0$
\begin{equation}
\begin{split}
    \bigg|&\int_{\partial \Omega(r)} \frac{\partial \theta}{\partial t} E^*\dd \ell + \int_{\partial \Omega^*(r)} \frac{\partial \theta^*}{\partial t} E^*\dd \ell\bigg|\\
    & \leq C \int_{\m S^1}\bigg| \frac{1}{1-r} \frac{\partial \theta(\gamma_r(p))}{\partial t}\bigg|.\bigg\Vert\frac{\partial\gamma_r}{\partial p}\bigg\Vert + \bigg|\frac{1}{1-r}.\frac{\partial \theta^*(\gamma^*_r(p))}{\partial t}\bigg|.\bigg\Vert\frac{\partial \gamma^*_r}{\partial p} \bigg\Vert \dd p\\
    &\leq \frac{1}{1-r} \int_{\m S^1} \bigg|\frac{\partial \theta(\gamma_r(p))}{\partial t}\bigg|. \bigg\Vert\frac{\partial \gamma_r}{\partial p}(p) - \frac{\partial \gamma^*_r}{\partial p}(p)\bigg\Vert \\
    & \hspace{60pt} + \bigg|\frac{\partial \theta(\gamma_r(p))}{\partial t} + \frac{\partial \theta^*(\gamma^*_r(p))}{\partial t}\bigg|. \bigg\Vert\frac{\partial \gamma^*_r}{\partial p}\bigg\Vert \dd p
\end{split}
\end{equation}
goes to $0$ as $r\rightarrow1^-$ uniformly in $t$.

Using the parameters of Proposition \ref{prop:finitevolume}, we can see that the $t$ derivatives of the functions $a_\O,a_{\O^*}$ in the proof of Proposition \ref{prop:finitevolume} agree as well up to order 2, so by the same argument $\lim_{r\rightarrow 1^-} \delta V_1(r,t) = 0$.

As for any $r$ near $1^-$ then $\delta V(\gamma_t) = \delta V_1(r,t)+ \delta V_2(r,t)$, we send $r$ to $1$ on the right hand side to obtain the result.
\end{proof}

\subsection{Variation of mean curvature and Schl\"afli formula}

The goal of this section is to prove the identity between the renormalized volume and the universal Liouville action when the curve is regular enough (Corollary~\ref{cor:smooth_identity}).

The following result is proved by Krasnov--Schlenker, see \cite[Cor. 6.2]{KrasnovSchlenker_CMP}, for the renormalized volume of convex co-compact manifolds. 
We adapt it to the renormalized volume associated with a smooth Jordan curve using Theorem~\ref{thm:Schl\"afliformula}.

\begin{thm} \label{thm:first_var_VR}
The first order variation of the $V_R$
$$\delta V_R (\g)= - \frac{1}{4} \int_{\O \cup \O^*} \delta \hat H + \frac 12 \brac{\delta \hat \I, \hat  \II_0} \dd a^*$$
where $\hat \II_0 =  \vartheta \,\dd z^2 +  \bar \vartheta \, \dd \bar z^2$ is the traceless part of $\hat \II$, $\brac {A, B}$ stands for $\tr [\hat \I^{-1} A \hat \I^{-1} B]$.
\end{thm}

\begin{proof}
By Definition \ref{df:RenormalizedVolume}, Remark \ref{remark:VRdefinition} and Theorem \ref{thm:Schl\"afliformula},  we can express $\delta V_R$ as the integral of smooth 2-forms in $\O, \O^*$, so that at points where the respective Epstein--Poincar\'e maps are immersions these forms are given by the pullback of the form
\[(\delta H + \frac14\brac{ \delta\I, \II }) \,\dd a - \frac12 (\delta H \dd a - H\delta(\dd a)) 
\]
by the respective Epstein--Poincar\'e map. Following \cite[Section 6]{KrasnovSchlenker_CMP} this pullback is expressed precisely as $-\frac 14 \left(\delta \hat H + \frac 12 \brac{\delta \hat \I, \hat \II_0} \right)\dd \hat a$. As points where the Epstein--Poincar\'e maps are immersions are dense in $\O, \O^*$, and by the piecewise regularity of $E$ and of its shape operator (Lemma~\ref{lem:formextension}) then all forms discussed vary continuously, and the result follows.
\end{proof}

More explicitly, we can write the variation of $V_R$ in terms of the Beltrami differentials.
We consider a $C^{5,\a}$ family of Jordan curves $(\g_t)$ as in the previous section and let $F_t$ be the corresponding homeomorphism of $\Chat$ which maps $\O_0$ conformally onto $\O_t$ and a diffeomorphism from $\O_0^*$ to $\O_t^*$, as constructed in Section~\ref{sec:Decomposition}.
For $z \notin \g_0$, let $$\mu_t := \frac{\partial_{\bar z} F_t}{\partial_z F_t} = t \dot \nu + O (t^2).$$
In particular, 
$\d F = \frac{\dd} {\dd t} F_t |_{t = 0}
$
satisfies
$$\partial_{\bar z} (\d F )= \dot \nu, \qquad F_t (z) = z + t (\d F(z)) + O (t^2).$$
Since $F_t$ is conformal in $\O_0$, $\dot \nu|_{\O_0} \equiv 0$.

\begin{lemma}
     $\norm{\dot  \nu}_\infty < \infty$. Moreover, $\dot \nu|_{\O^*} \in H^{-1,1} (\O^*) \oplus \mf N (\O^*)$.
\end{lemma}
\begin{proof}
On $\Omega$ the 1-parameter family $F_t$ is conformal, while in $\overline{\Omega^*}$, $F_t|_{\O^*}$ is jointly $C^{3,\alpha}$ in $(t,z)$.  
The $L^\infty$ bound of $\dot \nu$ follows from the compactness of the domains (viewed in $\Chat$).

For the second claim, as $(\g_t)$ corresponds to a differentiable path in $T_0(1)$, the projection of $\dot \nu$ onto harmonic Beltrami differentials $\O^{-1,1} (\O^*)$ parallel to $\mf N(\O^*)$ lies in $H^{-1,1}(\O^*)$. This completes the proof.
\end{proof}

\begin{cor}\label{cor:variation_VR_mu}
The first variation of the renormalized volume associated with the family of deformed Jordan curves $(\g_t := F_t (\g_0))$ is given by 
    \begin{equation}\label{eq:var_V_R_Sch}
    \delta V_R(\g) =  -\Re \int_{\O^*}  \dot \nu \mc S [g^{-1}] \dd^2 z.
    \end{equation}
   where we recall  $g : \m D^* \to \O^*$ is any conformal map.
\end{cor}
\begin{proof}
As $\dot \nu \in L^\infty(\O^*)$ and  $\mc S [g^{-1}]$ extend continuously to the boundary, the integral  in \eqref{eq:var_V_R_Sch} is absolutely convergent. From Theorem~\ref{thm:first_var_VR}, we only need to check the pointwise identity 
\begin{equation}\label{eq:pointwise_id_hyp}
\left(\frac{1}{4} \delta \hat H + \frac{1}{8} \brac{\delta \hat \I , \hat \II_0}\right) \dd \hat a = \Re \left( \dot \nu \mc S[g^{-1}] \right)\, \dd^2 z
\end{equation}
on $\O^*$.
As
$$\dd F_t (z) = \dd z + t \partial_z (\d F) \,\dd z  + t \partial_{\bar z} (\d F) \,\dd \bar z + O(t^2) = \dd z + t \partial_z(\d F) \,\dd z  + t \dot \nu \,\dd \bar z + O(t^2) $$
and in the $\dd z, \dd \bar z$  coordinates
$$\dd F_t (z) \dd \overline{ F_t (z)} =  \begin{pmatrix}
t \bar {\dot \nu} &  \frac{1}{2} (1 + 2  t \Re (\partial_z (\d F)))  \\
\frac{1}{2} (1 + 2 t \Re(\partial_z (\d F))) & t \dot \nu 
\end{pmatrix} + O(t^2). $$

 Therefore, the hyperbolic metric in $\O_t^*$ is
\begin{align*}
   &  e^{\varphi }  (1 + 2 t s + O(t^2))\,\dd F_t (z) \dd \overline{ F_t (z)} \\
     &= \hat \I + t e^{\varphi} \begin{pmatrix}
 \bar {\dot \nu} &   \Re( \partial_z (\d F)) + s  \\
 \Re( \partial_z (\d F))  + s &  \dot \nu 
\end{pmatrix} + O(t^2). 
\end{align*}
 where $s$ is some smooth function on $\O^*$ and
$$\hat \I = e^{\varphi} \dd z \dd \bar z =  \frac 12
\begin{pmatrix}
0 & e^{\varphi} \\
e^{\varphi} & 0
\end{pmatrix}.$$
We obtain
$$\delta \hat \I = e^\varphi \begin{pmatrix}
 \bar {\dot \nu} &   \Re( \partial_z (\d F)) + s  \\
 \Re( \partial_z (\d F))  + s &  \dot \nu 
\end{pmatrix}.$$
Recall that 
$$
\hat \II_0 =  \begin{pmatrix}
 \vartheta & 0 \\
0 &  \bar \vartheta
\end{pmatrix} = \begin{pmatrix}
 \mc S (g^{-1}) & 0 \\
0 &  \overline{ \mc S (g^{-1})}
\end{pmatrix},
$$ 
giving
$$\brac{\delta \hat \I, \hat \II_0}  =  8 e^{-\varphi} \Re (\dot \nu \mc S [g^{-1}]).$$
Since $\dd \hat a = e^{\varphi} \,\dd^2 z$ and from Corollary~\ref{cor:epstein_curv},  $\hat H = - \hat K \equiv 1$ which implies $\delta \hat H \equiv 0$,
we obtain the claimed formula \eqref{eq:pointwise_id_hyp}. 
\end{proof}

\begin{cor}\label{cor:smooth_identity}
   For all $C^{5,\a}$ Jordan curves $\g$, 
    $$\tilde \Liouville (\g) = 4 V_R(\g).$$
\end{cor}
\begin{proof}
When $\g$ is a circle, then $\tilde \Liouville (\g) = 0$ and $V_R(\g) = 0$ since both Epstein surfaces are the geodesic plane bounded by $\g$.

Given a $C^{5,\a}$ Jordan curve $\g$. We consider a $C^{5,\a}$ family $(\g_t)_{t \in [0,1]}$ of Jordan curves as in Example~\ref{ex:deformation} such that $\gamma_0 = \partial D$ is a circle and $\g_1 = \g$. 
The variational formula Theorem~\ref{thm:S_1_first_variation} and Corollary~\ref{cor:variation_VR_mu} show that  $$\tilde \Liouville (\g) = 4 V_R(\g)$$
since $\tilde \Liouville (\g_0) = 4 V_R(\g_0)$.
\end{proof}

\subsection{Approximation of  general Weil--Petersson quasicircle} \label{sec:approximation}

The goal of the section is to prove the following theorem using an approximation.
\begin{thm}\label{thm:general_ineq}
    For any Weil--Petersson quasicircle $\g$,
    $$\tilde \Liouville (\g) \ge 4 V_R(\g).$$
\end{thm}
\begin{remark}
We have already proved the equality when $\g$ is a $C^{5,\a}$ Jordan curve. We also believe the equality holds for arbitrary Weil--Petersson quasicircle but are only able to prove the inequality. 
\end{remark}

For the inequality, we will use the approximation using equipotential curves.
Let $\g$ be a Weil--Petersson quasicircle bounding Jordan domain $\Omega$ and $f : \m D \to \O$  a univalent map uniformizing $\Omega$. Up to post-composing $f$ by a M\"obius map, we may assume that $f(0) = 0$, $f'(0) = 1$ and $f''(0) = 0$.
The equipotentials 
$$\g_n = f_n (\m S^1), \quad \text{where } f_n (z): = \frac {n}{n-1} f \left(\frac {n-1}{n} z \right)$$
form a family of analytic Jordan curves. The map $f_n$ satisfies the same normalization as $f$ at $0$. 
 We let $\O_n^* := \Chat \smallsetminus \overline {f_n (\m D)}$ (resp. $\O^* := \Chat \smallsetminus \overline {f (\m D)}$) and $g_n$ (resp. $g$) be an arbitrary conformal map $\m D^* \to \O_n^*$ (resp, $\m D^* \to \O^*$).
Apart from the analyticity, the family of equipotentials is nice because of the following theorem.

\begin{thm}[See {\cite[Cor.\,1.5]{VW1}}]\label{thm:equipotential}
Along the family of equipotentials, the universal Liouville action converges and is non-decreasing.
Specifically, $$\lim_{n \to \infty} \uparrow \tilde \Liouville (\g_n) = \tilde \Liouville (\g). $$
If $\g$ is not a circle, then $\tilde \Liouville (\g_{n+1}) > \tilde \Liouville (\g_{n})$. 
\end{thm}

\begin{lemma}\label{lem:conv_HdA}
\begin{equation}
\lim_{n\rightarrow \infty}\int_{\S_{\O_n} \cup \S_{\O^*_n}} \hspace{-10pt}H \dd a = \int_{\S_{\O} \cup \S_{\O^*} } \hspace{-10pt}H \dd a.
\end{equation}
\end{lemma}
\begin{proof}
It follows from \cite[Cor.\,A.4., Cor.\,A.6]{TT06} that the element $[\mu_n]$ in $T_0(1)$ associated with $\g_n$ converges to $[\mu]$ which is associated with $\g$. In particular, \cite[Chap.\,I, Thm.\,2.13, Thm.\,3.1]{TT06} implies that 
$$\int_{\m D} \norm{\mc S (f_n)}_{\m D}^2 \,\rho_{\m D} \, \dd^2 z = \int_{\m D} |\mc S (f_n)|^2 \rho^{-1}_{\m D} \,\dd^2 z \xrightarrow[]{n \to \infty} \int_{\m D} \norm{\mc S (f)}_{\m D}^2 \, \rho_{\m D} \, \dd^2 z.$$
As $T_0(1)$ is a topological group, $[\mu_n]^{-1}$ converges to $[\mu]^{-1}$ which implies
$$\int_{\m D^*} \norm{\mc S (g_n)}^2_{\m D^*}\, \rho_{\m D^*} \,\dd^2 z = \int_{\m D^*} |\mc S (g_n)|^2 \rho^{-1}_{\m D^*} \,\dd^2 z \xrightarrow[]{n \to \infty} \int_{\m D^*} \norm{\mc S (g)}^2_{\m D^*}\, \rho_{\m D^*} \,\dd^2 z.$$
The proof is completed using Theorem~\ref{thm:PE} and that $\norm{\mc S(f^{-1}) (f(\z))}_{\O} =  \norm{\mc S(f) (\z)}_{\m D}$.
\end{proof}

\begin{lemma} Recall that $V_{2} (\g)(\vare)$ denotes the signed volume between $\Ep_{\O}$ and $\Ep_{\O^*}$ above Euclidean height $\vare$.
Then $V_{2} (\g_n)(\vare)$ converges to $V_{2} (\g)(\vare)$ for all $\vare > 0$.
\end{lemma}

\begin{proof}
For this, we denote for $\vare > 0$,
$$K_{\vare,n} := \{\z \in \m D \colon \xi_n \circ f_n (\z) \ge \vare\}, \quad K_{\vare} := \{\z \in \m D \colon \xi \circ f (\z) \ge \vare\},  $$
where $(Z_n, \xi_n)$ is the Epstein--Poincar\'e map on the domain $\O_n = f_n (\m D)$ following the notations in Section~\ref{sec:explicit}.

By Corollary \ref{cor:bdyvals}
 \begin{equation*} 
\frac{\dist (f_n(\z), \g_n)}{5}  \le  |\xi_n \circ f_n (\z)| \le 4 \dist (f_n(\z), \g_n)
 \end{equation*}
 which implies for all $\z \in K_{\vare, n}$,
$$\dist (f_n(\z), \g_n) \ge \vare / 4. $$
 
 Since $f_n$ converges uniformly to $f$ on $\overline{\m D}$ from the explicit expression, 
 the derivatives of $f_n$ converge to the derivatives of $f$ uniformly on compact sets of $\m D$ by Cauchy's integral formula.

  Hence, there exists $n_0$ such that for all $n \ge n_0$, 
 \begin{equation*}
 \norm{f_n - f}_{\infty, \overline{\m D}} < \vare /16.  
 \end{equation*}
 This implies 
 $$\dist (f(\zeta), \g) \ge \vare / 8 \quad
\text{ and } \quad \xi \circ f (\z) \ge \vare/40.$$
Summarizing, for all $n \ge n_0$,
$$K_{\vare, n} \subset K_{\vare/40}.$$
Since $K_{\vare/40}$ is a compact set in $\m D$ independent of $n$,  
all derivatives of $f_n$ converge uniformly to the derivatives of $f$ on $K_{\vare/40}$.
As the Epstein--Poincar\'e map only depends on $f$, $f'$, and $f''$ (Theorem~\ref{thm:basic_epstein}), $\Ep_{\O_n} \circ f_n$ converges uniformly to $\Ep_{\O} \circ f$ uniformly on $K_{\vare/40}$. 
Similarly argument applies to the Epstein--Poincar\'e maps $\Ep_{\O_n^*} \circ g_n$.
We obtain that $V_{2} (\g_n)(\vare)$ converges to $V_{2} (\g)(\vare)$.
\end{proof}

We obtain the following corollary.
\begin{cor} \label{cor:finite_v_general}
If $\g$ is a Weil--Petersson quasicircle, then
     $$V(\g)\le    \frac14 \tilde \Liouville (\g) + \frac 12 \int_{\S_{\O} \cup \S_{\O^*}} \hspace{-10pt} H \dd a  < \infty.$$
\end{cor}

\begin{proof}
     For small enough $\vare >0$,
\begin{align*}
V_{2} (\g)(\vare)& = \lim_{n \to \infty} V_{2} (\g_n) (\vare)  \\
&\le  \lim_{n \to \infty} \frac14 \tilde \Liouville (\g_n) + \frac 12 \int_{\S_{\O_n} \cup \S_{\O^*_n}} \hspace{-10pt} H \dd a  = \frac14 \tilde \Liouville (\g) +\frac 12 \int_{\S_{\O} \cup \S_{\O^*}} \hspace{-10pt} H \dd a
\end{align*}
by Theorem~\ref{thm:equipotential} and Lemma~\ref{lem:conv_HdA}. 
We obtained the inequality by taking $\vare \to 0$.
\end{proof}
Theorem~\ref{thm:general_ineq} follows immediately from this corollary.

\section{Gradient flow of the universal Liouville action} \label{sec:gradient}

Following Bridgeman--Brock--Bromberg \cite{bridgeman2021weilpetersson} and Bridgeman--Bromberg--Vargas-Pallete \cite{BridgemanBrombergVargas}, we introduce the following flow on $T(1)$. For $[\mu] \in T(1)$, there is a natural isomorphism $T_{[\mu]}T(1) \simeq \Omega^{-1,1}(\D^*)$.
We therefore define the vector field
$$V([\mu]) := -4\frac{\overline{\mc S(g_{\mu}})}{\rho_{\D^*}} \in \Omega^{-1,1}(\D^*),$$
where $g_\mu$ is a conformal map defined on $\m D^*$ associated with $[\mu]$ as defined in Section~\ref{sec:universal_T}.

\begin{thm}\label{thm:gradient}
The vector field $V$ has flowlines that exist for all time on  $T(1)$. The flow is restricted to a flow on $T_0(1)$ and is the \textnormal(negative\textnormal) Weil--Petersson gradient of the Liouville functional $\Liouville$. Furthermore, all flowlines on $T_0(1)$ converge to the origin $[0]$, which corresponds to the round circle.
\end{thm}

\begin{proof} 
By the Nehari bound, in the Teichm\"uller metric on $T(1)$, $||V||_\infty \leq 6$.
Thus as $T(1)$ is complete in the Teichm\"uller metric, the flow under $V$ exists for all time on $T(1)$. If $[\mu] \in T_0(1)$ then by the characterization \eqref{eq:wp_schwarzian} 
$$\int_{\D^*} |\mc S(g_{\mu})|^2\rho_{\D^*}^{-1} \,\dd^2 z< \infty.$$
Thus $V([\mu]) \in H^{-1,1}(\D^*) \simeq T_{[\mu]}T_0(1) $ and therefore by integrability the flow preserves $T_0(1)$.
Furthermore if $\dot \nu \in H^{-1,1}(\D^*) \simeq T_{[\mu]}T_0(1) $ then  by  Theorem~\ref{thm:S_1_first_variation},
$$(\dd \Liouville)_{[\mu]}(\dot \nu) = 4 \Re \int_{\m D^*} \dot \nu \mc S(g_\mu) = -\Re \int_{\m D^*} \dot \nu \,\overline {V([\mu]) } \, \rho_{\m D^*} = -\brac{V([\mu]), \dot \nu}_{\WP}.$$
Therefore $\nabla_{\WP}\Liouville = - V$ and
\begin{equation}\label{eq:first_der_S_gradient}
     \dd \Liouville(V) = - ||V||^2_{\WP}.
\end{equation}

We consider the flowline $\m R_+ \to T_0(1):  t \mapsto \a (t)$ for $V$ starting at a point $[\mu] = \a (0) \in T_0(1)$. Since  $\Liouville \geq 0$, for all $T >0$,
$$0 \leq \int_0^T ||V(\alpha(t))||_{\WP}^2 \, \dd t = \Liouville([\mu]) - \Liouville(\alpha(T)) \leq \Liouville([\mu]).$$
Thus 
$$ \int_0^\infty ||V(\alpha(t))||_{\WP}^2 \, \dd t < \infty.$$
We therefore have a sequence $t_n \rightarrow \infty$ such that
$$\lim_{n\rightarrow \infty} ||V(\alpha(t_n))||_{\WP} = 0.$$
Therefore
the conformal maps $g_{\alpha(t_n)}$ satisfy 
$\norm{\mc S(g_{\alpha(t_n)})}_2 \rightarrow 0$. 
From the Hilbert manifold structure of $T_0(1)$,  see \cite[Ch.\,1, Def.\,2.11]{TT06}, this implies 
$\alpha(t_{n})$ converges in $T_0(1)$ to the origin $[0]$. In particular, $\Liouville (\alpha(t_{n})) \to 0$.

To show that the flow line converges to $[0]$ (not only along a subsequence), we note first that the Liouville action of $\a (t)$ is decreasing, which implies $\Liouville (\alpha(t)) \to 0$ as $t \to \infty$. 
We now show that this implies the convergence of the flow line in $T_0(1)$ to the origin. 

In fact, assuming the opposite, there is $\vare >0$ and a sequence $\alpha(t_{k})$ such that the Weil--Petersson distance to $0$ is greater than $\vare$ along the sequence. Let $\g_{k}$ be the quasicircle passing through $1,-1, -\ii$ associated with $\alpha(t_{k})$ (see Section~\ref{sec:T_1}), since their Liouville action is uniformly bounded, they are all $K$-quasicircles (see, e.g., \cite[Prop.\,2.9]{RW}) for some $K >1$ (namely, image of $\m S^1$ of a $K$-quasiconformal homeomorphism $\varphi_k$ of $\Chat$ fixing $1,-1, -\ii$). We can extract from the normal family $\{\varphi_k\}$ a subsequence $\varphi_{k(n)}$ which converges uniformly on $\m S^1$ as $n \to \infty$. 
We write $\g_\infty$ for the image of $\m S^1$ of the limiting map $\lim_{n\to \infty } \varphi_{k(n)}$. We have $\g_{k(n)}$ converges uniformly to $\g_\infty$.
From \cite[Lem.\,2.12]{RW}, we know the Liouville action is lower-semicontinuous, this implies
\begin{equation}\label{eq:lower_semi}
 0 = \liminf_{n\to \infty}\Liouville (\a_{t_{k(n)}}) =  \liminf_{n\to \infty}\tilde \Liouville (\g_{k(n)}) \ge \tilde \Liouville (\g_{\infty}). 
\end{equation}
Hence $ \tilde \Liouville (\g_\infty) = 0$. As $\m S^1$ is the only zero of $\tilde \Liouville$, $\g_\infty  = \m S^1$ and the corresponding point in $T_0(1)$ is the origin $0$. To see  $\a_{t_{k(n)}}$ also converges in $T_0(1)$, consider the conformal maps $f_n : \m D \to D_n$ fixing $0$ and $g_n : \m D^* \to D_n^*$ fixing $\infty$ as in the definition of $\Liouville (\alpha_{t_{k(n)}})$, where $D_n$ and $D_n^*$ are respectively the bounded and unbounded connected component of $\m C \smallsetminus \g_{k(n)}$. 
From \eqref{eq:lower_semi} and Carath\'eodory theorem that 
$$\lim_{n \to \infty} \log \abs{\frac{f_n'(0)}{g_n'(\infty)}} = 0, \text{ which implies } \lim_{n \to \infty} \int_{\m D} \abs{\frac{f_n''(z)}{f_n'(z)}}^2 \dd^2 z = 0.$$
This shows $\alpha_{t_{k(n)}}$ converges to $0$ in $T_0(1)$ by \cite[Cor.\,A.4]{TT06} and contradicts the assumption of the sequence being $\vare$ distance away from $0$.
\end{proof}

Using the gradient flow, we may bound the Weil--Petersson distance between $[\mu]$ and $[0]$ by the universal Liouville action. 

We first recall the notations from Section~\ref{sec:universal_T}. For $\mu \in \O^{-1,1} (\m D^*)$, let $w_\mu$ be the quasiconformal map $\Chat \to \Chat$ that is invariant under the reflection $z\mapsto 1/\bar z$, fixes $-1,-\ii,1$, and has complex dilatation $\mu$ in $\m D^*$. Let $R_{[\mu]} \colon T(1) \to T(1)$ be the right multiplication by $\mu$. Let $\beta: T(1) \rightarrow A_\infty (\m D)$ be the Bers embedding, which restricts to $T_0(1)$ to a map $ T_0(1) \rightarrow A_2 (\m D)$ (see Section~\ref{sec:WP}).

The following results are proved by Takhtajan and Teo, which we summarize in the lemma below. 
\begin{lemma}[\!{\cite[Ch.\,1: Rem.\,2.4, Lem.\,2.5,  Cor. 2.6]{TT06}}]
  \label{lem:delta}  There exists $0 <\d <1$ such that for all $\mu \in \O^{-1,1} (\m D^*)$ with $\norm{\mu}_\infty <\d$,
    $$
    \abs{\frac{|\partial_z w_\mu(z)|^2}{(1 - |w_\mu(z)|^2)^2} - \frac{1}{(1-|z|^2)^2}} <  \frac{1}{(1-|z|^2)^2}.
    $$
    Moreover, for such $\mu$, the derivative $D_0(\beta\circ R_{[\mu]}):H^{-1,1}(\D^*) \rightarrow A_2(\D)$ is a bounded linear isomorphism with
$$||D_0(\beta\circ R_{[\mu]})(\nu)||_2 \leq 24 ||\nu||_{\WP} \qquad  ||\nu||_{\WP} \leq K||D_{0} (\beta \circ R_{[\mu]})(\nu)||_2 $$
where $K = \sqrt 2/(1-\d)^2$.
\end{lemma}
\begin{thm} \label{thm:bound_distance}
With the same constants $\d$ and $K$ as in Lemma~\ref{lem:delta}.
Let $0 < c < 2 \d \sqrt{4 \pi/3}
$, then for $[\mu] \in T_0(1)$, 
\begin{equation}\label{eq:dist_bound}
 c(\dist_{\WP}([\mu],[0]) - K c) \leq \Liouville ([\mu]) .
\end{equation}
\end{thm}
\begin{proof}
Let $t \mapsto \alpha(t)$ be the gradient flow line starting at $[\mu]$.
Assume first that $\norm{V([\mu])} \ge c$ and let $\tau$ be the first time $\norm{V(\alpha(t))}_{\WP} = c$. Then $\norm{V(\alpha(t))}_{\WP} > c$ for all $t< \tau$. Thus
\begin{align*}
   \Liouville([\mu]) - \Liouville(\alpha(\tau)) = \int_0^\tau \norm{V(\alpha(t))}_{\WP}^2 \,\dd t & \geq c\int_0^\tau \norm{V(\alpha(t))}_{\WP} \, \dd t \\
   & \geq  c \, \dist_{\WP}([\mu],\alpha(\tau)). 
\end{align*}
Therefore
$$\Liouville ([\mu]) \geq c(\dist_{\WP}([\mu],[0])-\dist_{\WP}(\alpha(\tau),[0])).$$
By  \cite[Ch.\,1, Lem.\,2.1]{TT06}, for all $\phi \in A_\infty(\D)$,
\begin{equation}\label{eq:bound_infty_by_wp}
||\phi||_\infty := \sup_{z \in \m D} \norm{\phi (z)}_{\m D }\leq \sqrt{\frac{3}{4\pi}}\sqrt{\int_{\m D}||\phi(z)||_{\m D}^2 \,\rho_{\m D} \, \dd^2 z} =  \sqrt{\frac{3}{4\pi}}||\phi||_{2}.
\end{equation}
Hence, since $\norm{V(\a(\tau))} = c$, 
$$\norm{V(\alpha(\tau)))}_{\infty} \le c \sqrt{3/4\pi} < 2 \d.$$ Therefore 
$$ \norm {\hat \beta ([\a (\tau)])}_\infty = \norm{\mc S(g_{\alpha(\tau)})}_\infty < \d /2 < 1/2$$
where $\hat \beta$ is the Bers embedding $T(1) \to A_\infty (\m D^*)$. 
As $\hat \beta(T_0(1)) = \hat \beta(T(1)) \cap A_2(\m D^*)$  the linear path 
$$\gamma (s)  := [s \tilde \mu], \quad \text{where} \quad \tilde \mu = -  \frac{2}{\bar z^4}\frac{\mc S(g_{\alpha(\tau)})}{\rho_{\m D^*}} \left(\frac{1}{\bar z}\right) \text{ satisfies } \norm{\tilde \mu}_{\m D, \infty} < \d $$ for $s \in [0,1]$ from $0$ to $\alpha(\tau)$ is in the ball of radius $\d$ of $T(1)$, and also in  $T_0(1)$ since by Ahlfors--Weill theorem 
$$ \hat \beta ([s \tilde \mu]) =   s \mc S(g_{\alpha(\tau)}) \in A_2 (\m D^*).$$
In $A_2(\D^*)$ this path has length $\norm{V(\alpha(\tau))}_{\WP} = c$. 
By Lemma~\ref{lem:delta}, the preimage of this path by $\hat \beta$ has a length less than $Kc$ and obtain \eqref{eq:dist_bound}.

If $\norm{V([\mu])}_{\WP} < c$, then the above argument shows that $\dist_{\WP} ([\mu], 0) \le Kc$, so \eqref{eq:dist_bound} holds trivially as $\Liouville \ge 0$.
\end{proof}

\section{Comparisons to minimal surfaces and convex hull}
\label{sec:comments}

Using Proposition \ref{prop:minimalsurface} and Proposition \ref{prop:smallminimalsurface}, we will answer a question of Bishop \cite{BishopQ} about how minimal surfaces and convex hulls relate to Epstein--Poincar\'e maps. This section is independent of the proofs in the rest of the paper.

Let us denote $\Sigma_\Omega(t)$ as the image of $\Ep_{e^{2t}\rho_{\Omega}} =:\Ep_{\O} (t)$ for $t\in\mathbb{R}$, with $\Sigma_\O = \Sigma_\O (0)$. The following proposition shows that any minimal surface in $\mathbb{H}^3$ with boundary $\gamma\subset \overline{\mathbb{C}}$ is between appropriate equidistant images $\Sigma_\Omega(t), \Sigma_{\O^*}(t)$.  (We recall that $\overline{\mathbb{C}}\smallsetminus\gamma = \O\cup\O^*$.)
For a conformal map $f:\mathbb{D}\rightarrow\Omega$ we write
$$\norm{\mc S(f^{-1})}_\infty : = \sup_{z \in \O} |\mc S(f^{-1}) (z)| \rho_{\O}^{-1} (z) = \sup_{\z \in \m D} |\mc S(f) (\z)| \rho_{\m D}^{-1} (\z) = \norm{\mc S(f)}_\infty. $$
Observe that while the definitions of $\norm{\mc S(f^{-1})}_\infty$ and $\norm{\mc S(f)}_\infty$ use the conformal map $f$, these only depend on the domain $\O$. Indeed, as pointed out in Section \ref{subsec:EPsimplyconnected}, the Schwarzian quadratic differential of $f^{-1}$ is the same for any conformal map $f : \m D \to \Omega$.

The following result says that after slightly pushing the Epstein--Poincar\'e surfaces towards the conformal boundary, they are disjoint from any minimal surfaces bounded by the same boundary.
\begin{prop}\label{prop:minimalsurface}
  Let $M\subset\mathbb{H}^3$ be a minimal surface so that $\partial_\infty M = \gamma$. Let $M_\Omega$ be the closure of the component of $\mathbb{H}^3\smallsetminus M$ with conformal boundary $\Omega$. Given a conformal map $f:\mathbb{D}\rightarrow\Omega$, let $t_0 = \frac12\log \left(\max\lbrace 1, 2\Vert\mc S(f)\Vert_\infty-1\rbrace\right)$. Then for any $t\geq t_0$, we have  $\Sigma_{\Omega}(t)\subset M_\Omega$.
\end{prop}
\begin{proof}
Recall that by the discussion at the start of Section \ref{subsec:EPsimplyconnected}, the principal curvatures at infinity associated with  $\rho_\Omega$ (see Theorem~\ref{thm:forms_infty}) are bounded below by $1-2\Vert\mc S(f)\Vert_\infty$ by \eqref{eq:Poincare_k_hat}. Similarly, taking any domain $\Omega'\subset\Omega$ bounded by an equipotential, the principal curvatures at infinity associated with $\rho_{\O}$ are bounded from below by
$$1- 2 |\mc S(f^{-1})(z)|\rho_{\O'}^{-1} (z) \ge 1- 2 |\mc S(f^{-1})(z)|\rho_{\O}^{-1} (z)  \ge 1-2\Vert\mc S(f)\Vert_\infty$$
as the uniformizing map associated with $\O'$ is given by the restriction of $f^{-1}$ to $\O'$, while the metric $\rho_{\O'}$ is pointwise greater than or equal to $\rho_\O$. 

We now show that $\Sigma_{\Omega'}(t)\subset M_\Omega$ for any $t\geq t_0$.
Fixing $\Omega'\subset\Omega$, define $t'=\inf\lbrace t\in\mathbb{R}\,|\, \Sigma_{\Omega'}(t)\subset M_\Omega\rbrace$. As  $\partial\Omega'\subset\Omega$, then $t'<+\infty$ and $\Sigma_{\Omega'}(t')\cap M$ is a non-empty compact subset of $\mathbb{H}^3$.
If we assume by contradiction that $t'>t_0$, then $e^{2t'}\rho_{\Omega'}$ has principal curvatures 
at infinity given by 
$e^{-2t'}$ times those associated with $\rho_{\Omega'}$ from Theorem~\ref{thm:forms_infty}, hence,
bounded strictly below by $-1$. In particular, this implies that $\Sigma_{\Omega'}(t')$ is an immersed surface by the last bullet point of Theorem~\ref{thm:forms_infty}. 

As the mean curvature at infinity $\hat H$ is the opposite of Gaussian curvature of the metric (see Corollary~\ref{cor:epstein_curv}), then the principal curvatures at infinity $\hat k_{1,2}$ of $e^{2t'}\rho_{\O'}$ satisfy $(\hat k_1+\hat k_2)/2 = e^{-2t'}<1$ at every point.  This implies that the mean curvature vector of $\Sigma_{\Omega'}(t')$, given by 
\begin{equation}\label{eq:mean_convex}
    \frac{1}{2} \left(\frac{1-\hat k_1}{1 + \hat k_1} + \frac{1-\hat k_2}{1 + \hat k_2}\right) = \frac{1-\hat k_1\hat k_2}{(1+\hat k_1)(1+\hat k_2)}>0
\end{equation}
times the outer normal to the associated horoball. 
This leads to a contradiction as at any tangent point between $\Sigma_{\Omega'}(t')$ and $M$, the mean curvature vector points in the wrong direction by the relative position of $\Sigma_{\Omega'}(t')$ and $M$. 
Therefore, we have shown that $t' \le t_0$ and $\Sigma_{\Omega'}(t)\subset M_\Omega$ for any $t\geq t_0$.

Since $\Sigma_\Omega(t)$ can be obtained as a limit of $\Sigma_{\Omega'}(t)$, the conclusion follows for $\Omega$.
\end{proof}

\begin{remark}\label{rmk:Sf<1}
    From Proposition \ref{prop:minimalsurface} if $\Vert \mc S(f)\Vert_\infty, \Vert \mc S(g)\Vert_\infty < 1$ (where $f, g$ are uniformization maps for $\O,\O^*$) then the minimal surface $M$ lies between the Epstein--Poincar\'e maps from $\O, \O^*$. 
    Moreover, $\Vert \mc S(f)\Vert_\infty, \Vert \mc S(g)\Vert_\infty < 1$ imply that each Epstein--Poincar\'e map is mean-convex by the same proof as \eqref{eq:mean_convex}. It follows from the strong maximum principle that if $M$ intersects $\S_\O$ or $\S_{\O^*}$. 
    Assuming $M$ intersects $\S_\O$, then $M$ has to locally agree with $\S_\O$. This implies that on $\O$, there exists an open neighborhood where $\hat k_1 = \hat k_2 = 1$, which is equivalent to that $\mc S(f)$ vanishes on an open set of $\m D$. This is only possible if $\g$ is a round circle. Hence, if $\g$ is not a round circle, then $\Vert \mc S(f)\Vert_\infty, \Vert \mc S(g)\Vert_\infty < 1$ implies that the minimal surface $M$ lies strictly between the Epstein--Poincar\'e maps from $\O, \O^*$. This gives an alternate proof of Proposition \ref{prop:disjoint} under the assumption $\Vert \mc S(f)\Vert_\infty, \Vert \mc S(g)\Vert_\infty < 1$. 
\end{remark}

We can also impose conditions on the curvatures of the minimal surface to obtain a similar conclusion as in Remark \ref{rmk:Sf<1}.

\begin{prop}\label{prop:smallminimalsurface}
  Let $M\subset\mathbb{H}^3$ be a minimal surface so that $\partial_\infty M = \gamma$. Denote by $M_\Omega$ the closure of the component of $\mathbb{H}^3\smallsetminus M$ with conformal boundary $\Omega$. Assume that at any point of $M$, the principal curvatures are strictly between $-1$ and $1$. Then for any $t\geq 0$,  $\Sigma_{\Omega}(t)\subset M_\Omega$.
\end{prop}
\begin{proof}
For $z\in\Omega$, let $H_z$ be the boundary of the smallest horoball $B_z$ centered at $z$ that intersects $M$. We define $\nu_M(z)$ as the metric tensor that coincides with the visual metric at $z$ seen from any point $x \in H_z$.
As $z\notin\partial_\infty M$, $\nu_M (z)$ is well-defined. The condition on principal
curvatures of $M$ implies that $H_z$ is tangent to $M$ at a unique point. Indeed, $H_z\cap M$ is an isolated set, and if it is not a singleton, then we would find an intrinsic geodesic segment of $M$ whose interior belongs to $\mathbb{H}^3\smallsetminus\overline{B_z}$, but its endpoints lie in $H_z$. Such a curve would have an interior point of geodesic curvature (as a curve in $\mathbb{H}^3$) greater than $1$
(see for instance the proof of \cite[Proposition 3.4]{FarreVP}). On the other hand, as a geodesic segment of the surface $M$ (which has principal curvatures strictly between $-1$ and $1$), this curve will have geodesic curvature (as a curve in $\mathbb{H}^3$) less than $1$ at every point, which is a contradiction.

Since the point of tangency of $H_z$ and $M$ is unique, it follows that $M$ is the Epstein surface associated with $\nu_M$, as it is the horosphere envelope for this metric. Theorem~\ref{thm:forms_infty} and Corollary~\ref{cor:epstein_curv} then imply that the Gaussian curvature of $\nu_M$ at $z$ is given by $-\frac{1+\lambda^2}{1-\lambda^2}\leq-1$, where $\pm\lambda$ are the principal curvatures of $M$ at the point of tangency. The Ahlfors--Schwarz lemma implies $\nu_M\leq e^{2t}\rho_\Omega$ for any $t\geq 0$, which in particular implies that $\Sigma_{\Omega}(t)\subset M_\Omega$.
\end{proof}

Finally, we observe that the Epstein--Poincar\'e surfaces are close to the convex hull as they remain in a small neighborhood of the convex hull.

\begin{prop}\label{prop:PEconvexcore}
Define $\varepsilon = \log \left(\max\lbrace 1, 2\Vert\mc S(f)\Vert_\infty-1, 2\Vert\mc S(g)\Vert_\infty-1\rbrace\right)$.
Then $\Sigma_\O$, $\Sigma_{\O^*}$ belong to $C_\varepsilon(\gamma)$, the $\varepsilon$ neighborhood of the convex hull $C(\gamma)$.
\end{prop}
\begin{proof}
Let $t > \varepsilon/2$.
We first show that, $\S_\O(t)$ is an immersed, strictly mean-convex surface similarly as in Proposition~\ref{prop:minimalsurface}. Indeed, Theorem~\ref{thm:forms_infty} shows that the principal curvatures at infinity of $\Sigma_\Omega(t)$ are bounded below by $e^{-2t}(1- 2\Vert\mc S(f)\Vert_\infty)$. Hence, as $t>\varepsilon/2$,  the curvatures at infinity of $\Sigma_\Omega(t)$ are strictly greater than $-1$, and therefore, $\Sigma_\Omega(t)$ is an immersed surface. Moreover, as the principal curvatures at infinity have arithmetic mean equal to $e^{-2t}<1$, the surface $\Sigma_\Omega(t)$ is mean-convex with respect to the normal vector field given by $\widetilde{\Ep}_\Omega (t)$ (the normal vector pointing outward from the horoball) as in \eqref{eq:mean_convex}.

\textbf{Claim:} For $t > \vare/2$ and any round disk $\overline{D}\subset \Omega$, we claim that $\Sigma_\Omega(t)$ lies in the component of $\mathbb{H}^3\smallsetminus \Sigma_D(t)$ whose boundary at infinity contains $\gamma$.

 The claim implies that for $t>\varepsilon/2$ and any round disk $\overline{D}\subset \O$, the surface $\Sigma_D(t)$ is disjoint from $\Sigma_\Omega(t)$. Proposition~\ref{prop:disjoint} then shows that $\Sigma_D(t)$ is also disjoint from $\Sigma_{\Omega^*}(t)$. Similarly, for disks in $\O^*$. Hence for $t>\varepsilon/2$, $\Sigma_\Omega(t)$ is contained in $C_t(\gamma)$.
As $\Sigma_\Omega(t)$ is obtained from $\Sigma_\Omega$ by flowing distance $t$ along the normal flow and $C_t(\gamma)$ is the $t$-neighborhood of $C(\gamma)$, then from $\Sigma_\Omega(t)\subseteq C_t(\gamma)$ it follows that $\Sigma_\Omega$ lies in $C_{2t}(\gamma)$. The proposition then follows by taking $t\rightarrow (\varepsilon/2)^+$.

Now we prove the claim. Define the (signed) distance function $d:\mathbb{H}^3\rightarrow\mathbb{R}$ to $\Sigma_D(t)$ so that $d^{-1}(\lbrace s\rbrace)=\Sigma_D(s+t)$. Considering $d\circ\Ep_\O$, we define $d_0=\sup\lbrace d(\Ep_\O(t) (z))\,|\, z\in\O \rbrace$. The claim is equivalent to show that $d_0\leq0$, so let us argue by contradiction and assume $d_0>0$.

As $\gamma$ does not intersect $D$, we have $\lim_{z\in\O,\, z\rightarrow\gamma}d(\Ep_\O(t)(z))=-\infty$ from which it follows that $d_0$ is realized at a point $z_0\in\O$. Hence the normal vector to $\Sigma_\Omega(t)(z_0)$  must be parallel to the normal vector to $\Sigma_{D}(d_0+t)$ at $\Ep_D(d_0+t)(z'_0)$ for some $z'_0\in D$, since otherwise we can produce $z_1$ close to $z_0$ so that $d(z_1)>d(z_0)$. 

Now we argue that the two normal vectors are equal.
Indeed, as $d_0>0$ and $t>\vare/2$, $\Sigma_D(d_0 + t)$, $\Sigma_{\O}(t)$ are strictly mean-convex at $\Ep_\O(t)(z_0)=\Ep_D(d_0 + t)(z'_0)$ with respect to the normal vectors $\widetilde{\Ep}_\O(t)(z_0)$ and  $\widetilde{\Ep}_D(d_0+t)(z'_0)$ respectively. So the relative position of $\Sigma_\O(t)$ and $\Sigma_D(d_0+t)$ implies that $\widetilde{\Ep}_\O(t)(z_0)=\widetilde{\Ep}_D(d_0+t)(z'_0)$ and in particular $z_0=z'_0$. 

We now compare the principal curvatures on $\S_\O(t)$ and $\S_D(d_0 +t)$ at $\Ep_\O(t)(z_0)=\Ep_D(d_0 + t)(z_0)$. Since the function $\hat k \mapsto (1-\hat k)/(1+\hat k)$ is decreasing for $\hat k>-1$, and the largest of the curvature at infinity of $\Sigma_\Omega(t)$ at any point is greater or equal than their arithmetic mean $e^{-2t}$, the surface $\Sigma_\O(t)$ has at least a principal curvature (Theorem~\ref{thm:forms_infty}) at $\Sigma_\Omega(t)(z_0)$ bounded above by 
$$\frac{1-\hat k_+(z_0)}{1+\hat k_+(z_0)} \leq \frac{1-e^{-2t}}{1+ e^{-2t}} = \tanh(t).$$ 
On the other hand, $\Sigma_D(d_0+t)$ has principal curvatures $\tanh(d_0+t)>\tanh(t)$, which is a contradiction to their relative position. Hence $d_0\leq 0$ and the claim is proven. 
\end{proof}

\begin{remark}\label{rem:ccc_monotonicity}
    As with Remark \ref{rmk:Sf<1},  Proposition \ref{prop:PEconvexcore} shows that if both  $\Vert\mc S(f)\Vert_\infty, \Vert\mc S(g)\Vert_\infty <1$, then the Epstein--Poincar\'e maps have image inside the convex hull. This is in contrast with the analogous result for convex co-compact hyperbolic 3-manifolds, where the condition $\Vert\mc S(f)\Vert_\infty, \Vert\mc S(g)\Vert_\infty <1$ is not required. This is because in the convex co-compact case, we can take a point at infinity where the conformal factor between the Poincar\'e and the Thurston metric is maximized (by the co-compact action in the boundary) and show that such maximum is bounded by $1$ (see for instance \cite[Section 3.7]{Schlenker13}),
    rather than argue with tangencies of surfaces as in the proof of Proposition~\ref{prop:PEconvexcore}.
\end{remark}

\begin{remark}

    We give an example where the Epstein--Poincar\'e surfaces are \emph{not} contained in the convex hull, hence the need to consider the $\varepsilon$-neighborhood in contrast with Remark~\ref{rem:ccc_monotonicity}.
   Denote by $\Omega_0$ the complement of the line segment $[0,1]$ in $\Chat$ and by $\Omega_n$ a sequence of domains bounded by equipotentials of $\Omega_0$ so that $\Omega_n\xrightarrow[]{n\rightarrow\infty}\Omega_0$. The convex hull of $\Omega_0$ is given by the half-plane whose conformal boundary is given by the segment. 
  By comparison with the Epstein--Poincar\'e maps associated with the half-planes contained in $\O_0$, it is easy to see that $\Ep_{\O_0}$ pierces through the convex hull. As
  $\Omega_n\xrightarrow[]{n\rightarrow\infty}\Omega_0$, $\Ep_{\Omega_n}$ is not contained in the convex hull of $\partial\Omega_n$ for $n$ sufficiently large.
\end{remark}

\section{Extending variational formula to non-immersed case}\label{gen-Schl\"afli}

The goal of this section is to extend the Schl\"afli formula to the case when the Epstein--Poincar\'e surfaces are not immersions and prove Theorem~\ref{thm:Schl\"afliformula}. In order to do so, we will need to generalize various parts of the proof of Schl\"afli found in \cite{Souam04}.

Let us recall the setup in Section~\ref{sec:Decomposition}. Let $(\gamma_t)_{t \in [0,1]}$ be a $C^{5,\alpha}$ family of Jordan curves ($\alpha>0$). Let $\Omega(r),A(r), \Omega^*(r)$ be the piecewise decomposition of $\Chat$  where we defined the family of piecewise smooth maps $E_{r,t}:\Chat\rightarrow \mathbb{H}^3$. We define the unit normal vector field $\vec n$ using $\widetilde \Ep_{\O_t}$ and $\widetilde \Ep_{\O^*_t}$, and on $A(r)$ we choose $\vec n$ to have positive vertical component as in Lemma~\ref{lemma:piecewiseshape}.

We first extend the notion of the shape operator.
\begin{lemma}
   There is a piecewise $C^{2,\alpha}$ family of linear maps $$B_{r,t}(p):\mathbb{R}^2 \rightarrow \vec n^\perp(E_{r,t}(p))$$ for $\{(r,t, p) \,|\, r \in (1-\vare, 1), t \in [0,1], p \in \O(r) \text{ or } A(r) \text{ or } \O^*(r)\}$, so that at any point where $E_{r,t}$ is an immersion, $B_{r,t}(p) v$ agrees with $B(D_p E_{r,t} (v))$, where $B$ is the shape operator of the image of $E_{r,t}$.
\end{lemma}

\begin{proof}
    For $\{z \in \O \ | \ \|\mc S(f^{-1})(z)\|_\O \neq 1\}\subseteq\Omega$ (and analogously for $\Omega^*$) and $v\in\mathbb{R}^2$, it follows from elementary differential geometry that $B_{r,t}(p) v$ satisfies

    \begin{equation}\label{eq:Bdefinition}
     B_{r,t}(p) v = -\left(D^i(p){\vec n}^j(p)\Gamma_{i,j}^k(p) + \frac{\partial {\vec n}^k}{\partial v}(p)\right)e_k(p),
    \end{equation}
    where we are using Einstein's notation, $e_1(p), e_2(p), e_3(p)$ is the canonical base for $T\mathbb{H}^3$ at $\Ep_\O(p)$, $\Gamma_{i,j}^k$ its Christoffel symbols, $\vec n = {\vec n}^ie_i$ are the coordinates of the normal vector $\vec n =\widetilde{\Ep_\O}$ and $D^ie_i$ the coordinates of $D\Ep_\O(v)$.
    As the right-hand side of \eqref{eq:Bdefinition} is well-defined along $\{z \in \O \ | \ \|\mc S(f^{-1})(z)\|_\O = 1\}$, we use \eqref{eq:Bdefinition} to define $B_{r,t}(p) v$. Usually, the right-hand side of \eqref{eq:Bdefinition} is denoted by $-\frac{D\vec n}{dv}$, where $\frac{D\vec n}{dv}$ is the \emph{covariant derivative} of the vector field $\vec n$ along the parametrization $\Ep_\O$. By abuse of notation, we will still denote $\frac{D\vec n}{dv}$ as $\nabla_{D_pEv}\vec n$, even though this only holds if $E$ is an immersion at $p$.

    For the region $A(r)$, we can define $B_{r,t}$ by observing that the map $E_{r,t}$ is the composition of a smooth map into the horizontal lines described in step \ref{item:horizontal step}. The union of these lines is immersed for $r$ sufficiently close to $1$ and hence has a well-defined shape operator. Hence we define $B_{r,t}$ as the pullback of such shape operator by $E_{r,t}$. As $B_{r,t}$ is defined as a pullback of a shape operator, it satisfies the analogous identity to \eqref{eq:Bdefinition}.

    It follows then by Lemma~\ref{lemma:piecewiseshape} that $B_{r,t}$ is piecewise $C^{2,\alpha}$ in $\{(r,t, p) \,|\, r \in (1-\vare, 1), t \in [0,1], p \in \O(r) \text{ or } A(r) \text{ or } \O^*(r)\}$, and agrees with the pullback by $E_{r,t}$ of the shape operator of the image of $E_{r,t}$ whenever $E_{r,t}$ is an immersion.
\end{proof}

\begin{remark}
    Observe that given that we are defining $B_{r,t}$ as a pullback along $A_r$, then whenever $E_{r,t}$ is not immersed at a point of $A_r$, $B_{r,t}$ will be the $0$ vector along the non-immersed direction. This differs from the behavior at $\O(r), \O^*(r)$. $E_{r,t}$ is not immersed at points of $\Omega(r), \Omega^*(r)$ where the curvatures at infinity are $- 1$, since the metric $$\I(X,Y) = \frac14 \hat \I\left((\id+\hat B)X,(\id+\hat B)Y\right)$$ will vanish precisely at directions $X$ (at infinity) whenever $\hat BX=-X$. We note that whenever $\hat B$ does not have eigenvalue $-1$ then $(\id+B)(\id+\hat B) = 2\id$. Therefore
    $$\I(BX, BY) = \frac{1}{4}\hat \I\left((\id+\hat B)BX,(\id+\hat B)BY\right) = \frac{1}{4}\hat \I\left((\id-\hat B)X,(\id-\hat B)Y\right).$$
    Hence for a given eigenvector $X$ ($\norm{X}_{\rho_\Omega}=1$) with principal curvature $k\neq-1$ we see that $\norm{BX}_{\m H^3} = \frac{|1-k|}{2}$, where this norm is given by the hyperbolic metric in $\mathbb{H}^3$. Hence for $k=-1$, $\norm{BX}_{\m H^3} = 1$ is a unit vector.
\end{remark}

\begin{remark}
    When there is no ambiguity, we will drop the subscripts $r,t$ in $E_{r,t}$ to simplify the notation.
\end{remark}

The following generalizes the key formula to prove the differential Schl\"alfi formula (see \cite[Proposition 5]{Souam04}). Let $\frac{\partial}{\partial t}|_{t=0}E_{r,t} = \xi$ be the piecewisely defined vector field by the first order variation on $t$, and let $\frac{D}{dt}$ denote the covariant derivative of a vector field along a curve in $\mathbb{H}^3$.

\begin{prop}\label{prop:geometricidentity} For any $p\in \Chat = \Omega(r)\cup A(r)\cup \Omega^*(r)$ and $u,v\in\mathbb{R}^2$ 
\begin{equation}\label{eq:PreSchl\"afliFormula} 
\brac{ \frac{D}{d\xi} (B (p) u),D_p E (v)} = - \brac{ \frac{D}{du}\frac{D}{d\xi} \vec n, D_pE (v)} + \brac{ R(\xi,D_pE (u))\vec n,D_p E (v) }
\end{equation}
where we follow the convention $R(X,Y)Z = \nabla_Y\nabla_X Z - \nabla_X\nabla_Y Z + \nabla_{[X,Y]}Z$, and $\brac{\cdot, \cdot}$ is the hyperbolic metric tensor in $\m H^3$.
\end{prop}
\begin{proof}
Recall the equality (all evaluated at $p$ and using the notation in \eqref{eq:Bdefinition})
\[B u = - \frac{D}{du}\vec n.
\]
Differentiating along $\xi$ and taking inner product with $DE(v)$
\begin{equation}
\begin{split}
   \brac{ \frac{D}{d\xi} Bu, DE(v) }  &= - \brac{ \frac{D}{d\xi}\frac{D}{du}\vec n,DE(v) } \\
    &= -\brac{ \frac{D}{du}\frac{D}{d\xi} \vec n,DE(v) } + \brac{ R(\xi,DE(u))\vec n,DE(v) },
\end{split}
\end{equation}
where we are using the curvature tensor to exchange the order of derivations.
\end{proof}

\begin{remark}
At points where $E$ is an immersion, we can write 
\[ \brac{ \nabla_\xi (Bu), DE (v) } = \brac{ B'(DE(u)), DE(v)} + \brac{ \nabla_{Bu} \xi, DE(v) }
\]
which is the formula appearing in \cite[Prop.~5]{Souam04}, where $B'$ is the covariant derivative with respect to $t$ of $B$ in the immersed surface image. At points where $E$ is an immersion \eqref{eq:PreSchl\"afliFormula} can be written as
\begin{align*}
  &\brac{ \nabla_\xi (B (p) u),D_p E (v)} \\
  &= - \brac{ \nabla_{D_p E (v)}\nabla_\xi \vec n, D_pE (u) }  + \brac{ R(\xi,D_pE (u))\vec n,D_p E (v) }  
\end{align*}
given the identification between covariant derivatives and connections.
\end{remark}

We recall that  $\dd a$ is the area form on the Epstein surface. We next extend the form $\brac{ \frac{D}{d\xi}(B\cdot) ,DE (\cdot)} \,\dd a$  to the non-immersed case.

Note that the trace of $\brac{ R(\xi,DE (\cdot))\vec n,DE  (\cdot) }$ is $-2\brac{\xi,\vec n}$ at immersion points, while the derivative of volume $V_2 (r,t)$ bounded by $E_{r,t}$ is given by 
\begin{equation}\label{eq:var_v_2_r_t}
     \frac{\partial}{\partial t}\bigg|_{t=0} V_2(r,t) = \int_{\Chat}-\brac{\xi,\vec n} E^*\dd a, 
\end{equation}
where the minus sign is because the normal vector points inwards.

Our goal is to show that if we take the trace in the remaining terms $\brac{ \nabla_\xi (B (p) \cdot),D_p E (v\cdot)}$, $ \brac{ \nabla_{D_p E (\cdot)}\nabla_\xi \vec n, D_pE (\cdot) }$ in (\ref{eq:PreSchl\"afliFormula}), integrate them against $E^* (\dd a)$ over $p$ and make $r\rightarrow1^-$ we obtain the right-hand side of the equation in  Theorem~\ref{thm:Schl\"afliformula}. As $E_{r,t}$ is not a piecewise immersion, our main concern is how to perform this trace for a non-immersion. The answer is that even though the trace is not well defined (as the metric degenerates), the trace times the area form $E^* (\dd a)$ extends to a piecewise differential for in $\Chat$. Let us first address this procedure in $\Omega(r), \Omega^*(r)$.

\begin{lemma}\label{lem:formextension}
    The $2$-form $\tr\brac{ R(\xi,DE (\cdot))\vec n,DE  (\cdot) }E^* \dd a$ defined on the set of immersion points 
    $\{z \in \O \ | \ \|\mc S(f^{-1})(z)\|_\O \neq 1\}$ extends as $C^{2,\alpha}$ differential $2$-form to the locus $\{z \in \O \ | \ \|\mc S(f^{-1})(z)\|_\O = 1\}$ as $- 2\brac{\xi,\vec n} E^*\dd a$. 
    
    Similarly the $2$-form $\tr \brac{ \nabla_{D_p E (\cdot)}\nabla_\xi \vec n, D_pE (\cdot) }E^*\dd a$ extends as $\dd ( i_{\nabla_\xi \vec n})$, where $i_{\nabla_\xi \vec n}$ is the $1$-form defined by $u\mapsto \brac{ D_pEu , \nabla_\xi \vec n}$ and $\dd$ denotes the exterior derivative. 
\end{lemma}

\begin{proof}
    As per the discussion before this Lemma, in the set of immersion points
    $$\{z \in \O \ | \ \|\mc S(f^{-1})(z)\|_\O \neq 1\}$$ 
    the $2$-form $\tr\brac{ R(\xi,DE (\cdot))\vec n,DE  (\cdot) }E^* \dd a$ agrees with $- 2\brac{\xi,\vec n} E^*\dd a$. As $\{z \in \O \ | \ \|\mc S(f^{-1})(z)\|_\O \neq 1\}$ is an open dense set of $\O$, then we can extend uniquely $\tr\brac{ R(\xi,DE (\cdot))\vec n,DE  (\cdot) }E^*\dd a$ as $- 2\brac{\xi,\vec n} E^*\dd a$ to all of $\O$, as $- 2\brac{\xi,\vec n} E^*\dd a$ is a well-defined $C^{2,\alpha}$ $2$-form in $\O$.

    Similarly, on the set $\{z \in \O \ | \ \|\mc S(f^{-1})(z)\|_\O \neq 1\}$ 
    \begin{align*}
        \tr\brac{ \nabla_{D_pE(\cdot)}\nabla_\xi \vec n, D_pE(\cdot) } E^*\dd a &= E^*(\operatorname{div}(\nabla_\xi \vec n)\dd a) = E^*(\dd \brac{ \cdot , \nabla_\xi \vec n}) \\
        &= \dd ( i_{\nabla_\xi \vec n}).
    \end{align*}
    As $\{z \in \O \ | \ \|\mc S(f^{-1})(z)\|_\O \neq 1\}$ is an open dense set in $\O$ and $-\dd ( i_{\nabla_\xi \vec n})$ is a well-defined $C^{2,\alpha}$ form in $\O$, then $- \tr\brac{ \nabla_{D_pE(\cdot)}\nabla_\xi \vec n, D_pE(\cdot) } E^*\dd a$ extends uniquely as a $C^{2,\alpha}$ $2$-form.
\end{proof}

Observe that by Proposition~\ref{prop:geometricidentity} and Lemma~\ref{lem:formextension}, 
at immersion points 
\begin{equation}\label{eq:TraceSchl\"afli}
    \tr\brac{ \nabla_\xi (B\cdot),D_pE\cdot} E^*\dd a = -\dd ( i_{\nabla_\xi \vec n}) - 2\brac{\xi,\vec n} E^*\dd a,
\end{equation}
which by Lemma~\ref{lem:Tr_variation_V} yields that in particular on $\{z \in \O \ | \ \|\mc S(f^{-1})(z)\|_\O \neq 1\}$ of $\O$,
\begin{equation}\label{eq:TraceSchl\"afli_imm}
    2E^*(\d H + \frac14\brac{ \d \I, \II }\,\dd a) = -\dd ( i_{\nabla_\xi \vec n}) - 2\brac{\xi,\vec n} E^*\dd a.
\end{equation}

\begin{proof}[Proof of Theorem~\ref{thm:Schl\"afliformula}]
We proceed as in the proof of Theorem~\ref{thm:schlafli_imm} by taking $r$ close to 1 and defining $V_1(r,t), V_2(r,t)$ as before, so that  $V(\gamma_t) = V_1(r,t)+V_2(r,t)$. In particular, $V_2(r,t)$ is defined as the volume bounded by $E_{r,t}$. Namely, extend $E_{r,t}:\Chat \rightarrow \mathbb{H}^3$ as a map from the closed ball $B^3$ so that
\[V_2(r,t) = \int_{B^3}E^*_{r,t}(\vol_{\mathbb{H}^3}).
\]

For $E_{r,t}$ in $A(r)$, we can establish and trace (\ref{eq:PreSchl\"afliFormula}) in the embedded surface that contains the image of $E_{r,t}$ (for $r$ sufficiently close to $1$) and then take the pullback by $E_{r.t}$.

By Stokes, this definition does not depend on the specific extension of $E_{r,t}$ to $B^3$. Since $E_{r,t}$ vary $C^{3,\alpha}$ as piecewisely defined map from $\Omega(r), \Omega^*(r), A(r)$, we can take the extension to vary $C^{3,\alpha}$ on $t$ and check that $\partial_tV_2(r,t)$ is given by
\[\partial_tV_2 = \bigg( \int_{\Omega(r)} + \int_{\Omega^*(r)} + \int_{A(r)} -\brac{ \xi, \vec n} E^*\dd a\bigg)
\]
where $\xi = \partial_tE_{r,0}$ and $\dd a$ is the area form of the orthogonal plane to $\vec n$. The negative sign is due to the fact that we are taking normal vector $\vec n$ pointing \textit{inward} the region bounded by $E_{r,t}$.

Applying (\ref{eq:TraceSchl\"afli}) then

\begin{equation*}
    \partial_tV_2 = \bigg( \int_{\Omega(r)} + \int_{\Omega^*(r)} + \int_{A(r)} \frac12 \tr\brac{ \nabla_\xi (B\cdot),D_pE\cdot} E^*\dd a + \frac12 \dd(i_{\nabla_\xi \vec n})\bigg).
\end{equation*}
Applying Stokes theorem for $\frac12 \dd ( i_{\nabla_\xi \vec n})$ yields the integral of $\frac12i_{\nabla_\xi \vec n}$ over each boundary component. Since $E^{r,t}$ is embedded along $\partial A(r)$, then as in \cite{Souam04}, along $\partial A(r)$, 
\begin{align*}
  i_{\nabla_\xi (\vec n^{\Omega(r)})} + i_{\nabla_\xi (\vec n^{A(r)})} &= \frac{\partial \theta^+}{\partial t} E^*\dd \ell\\
i_{\nabla_\xi (\vec n^{\Omega^*(r)})} + i_{\nabla_\xi (\vec n^{A(r)})}& = \frac{\partial \theta^-}{\partial t} E^*\dd \ell  
\end{align*}
where $\theta^+(x)$ (respectively $\theta^-(x)$) is the exterior dihedral angle of the planes orthogonal to $\vec n^{\Omega(r)}, \vec n^{A(r)}$ at $E(x)$ (respectively $\vec n^{\Omega(r)^*}, \vec n^{A(r)}$ at $E(x)$), and $\dd \ell$ is the length form in $\mathbb{H}^3$.

Applying then Stokes for $\partial_t V_2$ we get
\begin{align}\label{eq:var_V2}
       \begin{split}
           \partial_tV_2 = & \bigg( \int_{\Omega(r)} + \int_{\Omega^*(r)} + \int_{A(r)} \frac12 \tr\brac{ \nabla_\xi (B\cdot),D_pE\cdot} E^*\dd a\bigg) \\ 
    & +  \frac12\bigg( \int_{\partial \Omega(r)} \frac{\partial \theta^+}{\partial t} E^*\dd \ell + \int_{\partial \Omega^*(r)} \frac{\partial \theta^-}{\partial t} E^*\dd \ell\bigg).
       \end{split}      
\end{align}
This proves the analog of Theorem~\ref{theorem:Schl\"afli_pw} for the non-immersed case.
Then finally by applying \eqref{eq:TraceSchl\"afli_imm} on the open dense set where $\Ep_\O, \Ep_{\O^*}$ are immersions, taking $r\rightarrow1^-$ and proceeding as in the proof of Theorem~\ref{thm:schlafli_imm}, we obtain the desired formula.
\end{proof}

\bibliographystyle{abbrv}
\bibliography{mybib}

\end{document}